\documentclass{amsart}

\usepackage[centertags]{amsmath}
\usepackage{amsfonts}
\usepackage{amssymb}
\usepackage{amsthm}
\usepackage[colorlinks, hidelinks]{hyperref}
\usepackage{tikz}
\usepackage{tikz-cd}
\usepackage[normalem]{ulem}
\usepackage[shortlabels]{enumitem} 
\usepackage{marginnote}
\usepackage[all]{xy}

\usepackage{bm}

\newcommand{\N}{\mathbb{N}}
\newcommand{\R}{\mathbb{R}}

\newcommand{\Z}{\mathbb{Z}}

\newcommand{\PP}{\mathbb{P}}
\newcommand{\E}{\mathbb{E}}
\newcommand{\A}{\mathcal{A}}

\newcommand{\Natural}{\mathbb N}

\newcommand{\Real}{\mathbb R}

\newcommand{\abs}[1]{\left\vert#1\right\vert}
\newcommand{\set}[1]{\left\{#1\right\}}
\newcommand{\restricted}{\mathord{\upharpoonright}}

\newcommand{\ep}{\varepsilon}
\newcommand\restr[2]{{
		\left.\kern-\nulldelimiterspace 
		#1 
		\littletaller 
		\right|_{#2} 
}}
\newcommand{\littletaller}{\mathchoice{\vphantom{\big|}}{}{}{}}

\definecolor{egraf}{rgb}{0.2,0.4,0}

\newcommand{\norm}[1]{\left\Vert#1\right\Vert}
\newcommand{\duality}[1]{\left\langle#1\right\rangle}

\newcommand{\Free}{{\mathcal F}}
\newcommand{\Lip}{{\mathrm{Lip}}}

\newcommand{\F}{\mathcal{F}}              
\newcommand{\conv}{\mathop\mathrm{conv}}
\newcommand{\supp}{\mathop\mathrm{supp}}

\newcommand{\eps}{\varepsilon}

\def\<{\langle}
\def\>{\rangle}

\newcommand{\Id}{Id}

\newcommand{\DP}{{\mathrm{DP}}}
\renewcommand{\H}{{\mathcal{H}}}
\newcommand{\CF}{{\mathrm{CF}}}
\newcommand{\RN}{{\mathrm{RN}}}

\newcommand{\NN}{\mathbb{N}}                                
\newcommand{\RR}{\mathbb{R}}                                
\newcommand{\lipfree}[1]{\mathcal{F}({#1})}                 
\newcommand{\cl}[1]{\overline{#1}}                          
\DeclareMathOperator{\lip}{lip}                             

\renewcommand{\hat}{\widehat}

\theoremstyle{plain}
\newtheorem{thm}{Theorem}[section]

\newtheorem{cor}[thm]{Corollary}
\newtheorem{corollary}[thm]{Corollary}
\newtheorem{lem}[thm]{Lemma}
\newtheorem{lemma}[thm]{Lemma}
\newtheorem{prop}[thm]{Proposition}
\newtheorem{proposition}[thm]{Proposition}

\newtheorem*{KirchLemma}{Kirchheim's lemma}
\newtheorem*{KirchThm}{Theorem}

\theoremstyle{definition}

\newtheorem{definition}[thm]{Definition}

\newtheorem{remark}[thm]{Remark}

\newtheorem{example}[thm]{Example}
\newtheorem{quest}{Question}

\author[G. Flores]{Gonzalo Flores}
\address[G. Flores]{Universidad de O'Higgins, Instituto de Ciencias de la Ingenier\'ia, Av. Lib. Gral. Bernardo O'Higgins 611, Rancagua, Chile}
\email{gonzalo.flores@uoh.cl}

\author[M. Jung]{Mingu Jung}
\address[M. Jung]{June E Huh Center for Mathematical Challenges, Korea Institute for Advanced Study, 02455 Seoul, Republic of Korea}
\email{jmingoo@kias.re.kr}

\author[G. Lancien]{Gilles Lancien}
\address[G. Lancien]{Université Marie et Louis Pasteur, CNRS, LmB (UMR 6623), F-25000 Besançon, France.}
\email{gilles.lancien@univ-fcomte.fr}

\author[C. Petitjean]{Colin Petitjean}
\address[C. Petitjean]{Univ Gustave Eiffel, Univ Paris Est Creteil, CNRS, LAMA UMR8050, F-77447 Marne-la-Vallée, France}
\email{colin.petitjean@univ-eiffel.fr}

\author[A. Prochazka]{Anton\'in Prochazka}
\address[A. Prochazka]{Université Marie et Louis Pasteur, CNRS, LmB (UMR 6623), F-25000 Besançon, France}
\email{antonin.prochazka@univ-fcomte.fr}

\author[A. Quilis]{Andr\'es Quilis}
\address[A. Quilis]{Institute of Mathematics and Statistics, University of Tartu, Narva mnt 18, 51009 Tartu, Estonia}
\email{andresqsa@gmail.com}

\begin{document}
	\title[On curve-flat Lipschitz functions and their linearizations]{On curve-flat Lipschitz functions and their linearizations}
	


	\begin{abstract}
        We show that several operator ideals coincide when intersected with the class of linearizations of Lipschitz maps.
        In particular, we show that the linearization $\hat{f}$ of a Lipschitz map $f:M\to N$ is Dunford-Pettis if and only if it is Radon-Nikod\'ym if and only if it does not fix any copy of $L_1$. 
        We also identify and study the corresponding metric property of $f$, which is a natural extension of the curve-flatness.
        \end{abstract}

	\maketitle

\section{Introduction}

The Lipschitz-free space $\mathcal F(M)$ over a pointed metric space $M$ is a Banach space that is built around $M$ in such a way that $M$ is isometric to a subset $\delta(M)$ of $\mathcal F(M)$, and 0-preserving Lipschitz maps from $\delta (M)$ into any other Banach space $X$ uniquely extend to bounded linear operators from $\mathcal F(M)$ into $X$ (see Section~\ref{ss:LFprelim} for details and references).
It follows that, given two metric spaces $M,N$, for any 0-preserving Lipschitz map $f:M \to N$ there exists its canonical linearization $\widehat{f}:\mathcal F(M) \to \mathcal F(N)$ which is a bounded linear operator, called \emph{Lipschitz operator} hereafter.
The mapping $M \mapsto \mathcal F(M)$ and $f \mapsto \widehat{f}$ is a functor between the category of pointed metric spaces together with 0-preserving Lipschitz maps and the category of Banach spaces together with bounded linear maps. 
In order to gain a better understanding of either of these categories, it is intriguing to study both
\begin{itemize}
    \item[A)] the relation between the metric properties of $M$ and the linear properties of $\mathcal F(M)$, as well as  
    \item[B)] the relation between properties of $f$ and corresponding properties of $\widehat{f}$. 
\end{itemize}
For example, in the direction A) it has been shown in \cite[Theorem C]{AGPP}: \emph{For a complete $M$ {the following assertions are equivalent.}
\begin{itemize}
\item[(i)] $M$ is purely-1-unrectifiable (p-1-u), 
\item[(ii)] $\mathcal F(M)$ has the Radon-Nikodým property (RNP),
\item[(iii)] $\mathcal F(M)$ has the Schur property (SP),
\item[(iv)] $L_1$ is not isomorphic to a subspace of $\mathcal F(M)$. 
\end{itemize}}

In this paper, we aim to proceed in the direction B), i.e. we want to study the part of the free functor acting on morphisms.
Our main result is the following extension of the aforementioned Theorem C of \cite{AGPP}:
\begin{thm}\label{t:flagship}
    {Let $M$ and $N$ be pointed metric spaces and assume that $M$ is complete.} For a 0-preserving Lipschitz map $f:M\to N$, the following are equivalent:
\begin{itemize}
\item[(i)] $f$ is curve-flat (CF), that is, for every $K \subset \R$ compact and every Lipschitz map $\gamma : K \to M$,
	$$\lim_{\substack{y \to x \\ y\in K}}\frac{d(f(\gamma(x)),f(\gamma(y)))}{|x-y|} = 0 \quad \text{for $\lambda$-almost every $x \in K$, }$$
{where $\lambda$ denotes the Lebesgue measure on $\R$.}
\item[(ii)] $\widehat{f}$ is a Radon-Nikodým (RN) operator.
\item[(iii)] $\widehat{f}$ is a Dunford-Pettis (DP) operator.
\item[(iv)] $\widehat{f}$ does not fix any copy of $L_1$.
\end{itemize}
\end{thm}
We will prove Theorem~\ref{t:flagship} in Section~\ref{section:DP} (the equivalence of (i), (iii) and (iv)) and in Section~\ref{section:RN} (the equivalence of (i), (ii) and (iv)).

Theorem~\ref{t:flagship} is indeed an extension of \cite[Theorem C]{AGPP} because of the functoriality mentioned above (in particular since $\widehat{\Id_M}=\Id_{\Free(M)}$) and since we have
\begin{itemize}
\item[(i)] $M$ is p-1-u if and only if $\Id_M:M \to M$ is CF,
\item[(ii)] $X$ has the RNP if and only if $\Id_X:X \to X$ is RN,
\item[(iii)] $X$ has the SP if and only if $\Id_X:X \to X$ is DP,
\item[(iv)] $L_1$ is isomorphic to a subspace of $X$ if and only if $\Id_X:X \to X$ fixes a copy of $L_1$.
\end{itemize}

In the subsequent sections we explore some other operator properties which extend properties of Banach spaces through similar equivalences.
Namely, notice that it is true that $X$ has the RNP if and only if $\Id_X:X \to X$ is a strong Radon-Nikod\'ym operator (strong RN). 
In Section~\ref{section:RepresentableAndStrongRN} we study the possibility of adding the item ``$\widehat{f}$ is strong RN'' to the equivalence of Theorem~\ref{t:flagship}.
While we show that this is not possible in general (i.e. there exists curve-flat $f$ such that $\widehat{f}$ is not strong RN; see Example \ref{Example:RN_not_SRN}), we 
show that for compact countably 1 rectifiable $M$ such equivalence holds true (see Proposition~\ref{p:RN_eq_SRN_in_1_rectifiable_spaces}).

Also notice that $X$ has the SP if and only if $\Id_X:X \to X$ factors through a Banach space $Z$ with the Schur property.
We do not know whether the item ``$\widehat{f}$ factors through a Schur space'' could be added to the equivalence of Theorem~\ref{t:flagship}. 
In Section~\ref{section:factorization} we make several remarks in this respect.
In general it is not possible to add an item claiming ``$f:M\to N$ factors through a p-1-u complete metric space'', i.e. there exists curve-flat $f$ which does not factor through any complete p-1-u space (see Example~\ref{cor:Example_CF_not_Factp-1-u}). Notice that if each curve-flat function factored through a complete p-1-u metric space, Theorem~\ref{t:flagship} could be obtained by direct application of \cite[Theorem C]{AGPP}. 
Thus, the above example shows that such a naive approach does not work in general.
In fact, we characterize those compact metric spaces $M$ such that every curve-flat $f:M \to N$ factors through a p-1-u complete metric space (see Theorem~\ref{t:MainOfSection6}). 
It happens precisely when the second dual of the space of uniformly locally flat functions is the space of curve-flat functions. From a metric viewpoint, this happens exactly when collapsing all curve fragments in $M$ leads to a p-1-u space.

Finally, in an appendix, we discuss a short proof due to Sylvester Eriksson-Bique of Theorem~A in~\cite{AGPP} and its relation to our proof of Theorem~\ref{t:flagship}.

\section{Preliminaries}
Throughout the paper, we will use the letters $M,N$ for metric spaces. If $p \in M$, $r\geq 0$ and $S$ a non empty subset of $M$, we will use the notation:
\begin{align*}
	B(p,r) &=  \{x \in M \; | \; d(x,p) \leq r \}\\
	d(p,S) &= \inf\{d(x,p) \; : \; x \in S\}\\
    [S]_r &=\set{x \in M:d(x,S)\leq r}\\
    \text{rad}(S) &=\sup\set{d(0,x):x\in S}.
\end{align*}
The letters $X,Y$ are reserved for real Banach spaces. As usual, $X^*$ denotes the topological dual of $X$, and $\mathcal L(X,Y)$ is the space of all bounded operators from $X$ to $Y$. Additionally, $B_X$ and $S_X$ denote the closed unit ball and the unit sphere of $X$, respectively. If $S \subset X$ then $\conv(S)$ stands for the convex hull of $S$. 
Ordinarily, $L_1 = L_1([0,1])$ represents the Banach space of 
Lebesgue integrable functions from $[0,1]$ to $\R$. 
The Lebesgue measure {on $\R$} is denoted  
by $\lambda$.
Similarly, $\ell_1 = \ell_1(\N)$ is the space of absolutely summable sequences in $\R$.
To indicate that $X$ is isometrically isomorphic to $Y$, we use the notation $X \equiv Y$.

\subsection{Lipschitz-free spaces}\label{ss:LFprelim}
By a \emph{pointed metric space} we mean a metric space $M$ containing a distinguished point, denoted 0 and called a \emph{basepoint}. 

If $M$ and $N$ are pointed metric spaces then $\Lip_0(M,N)$ stands for the set of all Lipschitz functions $f :M \to N$ such that $f(0) = 0$. 
If $N=X$ is a real Banach space, then $\Lip_0(M,X)$ naturally becomes a Banach space when equipped with the Lipschitz norm:
$$ \forall f \in \Lip_0(M,X), \quad \|f\|_L = \sup_{x \neq y \in M} \frac{\|f(x)-f(y)\|_X}{d(x,y)}.$$
In the case $X = \R$, it is customary to shorten the notation to $\Lip_0(M):=\Lip_0(M,\R)$.
Next, for $x\in M$, we let $\delta_M(x) \in \Lip_0(M)^*$ be the evaluation functional defined by $\langle\delta_M(x) , f \rangle  = f(x)$, for all  $f$  in $\Lip_0(M)$. 
It is readily seen that $\delta_M: x \in M \mapsto \delta_M(x) \in \Lip_0(M)^*$ is an isometry. 
The \textit{Lipschitz-free space over $M$} is then defined as the closed linear span of all such evaluation functionals:
$$\F(M) := \overline{ \mbox{span}}^{\| \cdot  \|}\left \{ \delta(x) \, : \, x \in M  \right \} \subset \Lip_0(M)^*.$$
A classical reference on Lipschitz-free spaces is~\cite{Weaver2}.

Let us first recall the universal extension property of Lipschitz-free spaces. Let  $X$ be a Banach space and $f \in \Lip_{0}(M,X)$. Then there exists a unique $\overline{f} \in \mathcal{L}(\F(M),X)$ such that $f = \overline{f} \circ \delta_M$. That is, we have the following commutative diagram:  
$$ \xymatrix{
			M \ar[r]^f \ar[d]_{\delta_M}  & X.  \\
			{\F(M)} \ar[ur]_{\overline{f}} 
		} $$
		Moreover $\| \overline{f} \|_{\mathcal{L}(\F(M),X)} = \| f \|_L$.
In particular,
$\Lip_0(M,X) \equiv \mathcal{L}(\F(M),X)$ and $\F(M)$ is an isometric predual of $\Lip_0(M)$, that is $\Lip_0(M) \equiv \F(M)^*$. 

From the previous factorization, one easily deduces another fundamental feature of  Lipschitz-free spaces that we describe now. If $M,N$ are pointed metric spaces and $f : M \to N$ is a basepoint preserving Lipschitz map, then there exists a unique bounded linear operator $\widehat{f} : \F(M) \to \F(N)$ such that the following diagram commutes:
$$\xymatrix{
	M \ar[r]^f \ar[d]_{\delta_M}  & N \ar[d]^{\delta_N} \\
	\F(M) \ar[r]_{\widehat{f}} & \F(N).
}
$$
Moreover, we have that $\|\widehat{f}\| = \|f\|_L$. 
Operators of the form $\widehat{f}$ are commonly referred to as \textit{Lipschitz operators} or, more pragmatically, as linearizations of Lipschitz maps.
It is noteworthy that the adjoint of $\widehat{f} : \F(M) \to \F(N)$ is the ``pre-composition by $f$'' operator acting between the respective spaces of Lipschitz functions, that is $(\widehat{f})^* = C_f : g\in \Lip_0(N) \mapsto g \circ f \in \Lip_0(M)$. 

We also wish to recall some further important features of Lipschitz-free spaces. 
Firstly, note that, on bounded subsets of $\Lip_0(M)$,
the weak$^*$ topology coincides with the pointwise topology.
Next, if $0 \in K \subset M$, then $\F(K) \equiv \F_M(K)$ where: 
$$\F_M(K) := \overline{\text{span}} \{ \delta_M(x) \; | \; x \in K\} \subset \F(M).$$
Under this identification, the \textit{support of $\gamma \in \F(M)$} is the smallest closed subset $K \subset M$ such that $\gamma \in \F_M(K)$. 
It is denoted by $\text{supp}(\gamma)$. For further details, we refer to \cite{AP20, APPP2019}. If $x \neq y  \in M$, then the elementary molecule $m_{x, y} \in S_{\F(M)}$ is defined by:
$$ m_{x, y} := \frac{\delta(x) - \delta(y)}{d(x,y)}.$$

\subsection{Preliminaries on operators and operator ideals}\label{ss:OperatorIdeals}

In this paper we will deal with the following well known operator ideals:
\begin{itemize}[leftmargin=*, itemsep=5pt]
    \item {\bf Dunford-Pettis operators}. 
   An operator $T:X \to Y$ between Banach spaces is
    \textit{Dunford-Pettis} (or \textit{completely continuous}) whenever it is weak-to-norm sequentially continuous. 
    We will denote by $\DP$ the operator ideal in the sense of Pietsch~\cite{Pietsch} of Dunford-Pettis operators and by $\DP(X,Y)$ the set of all DP operators from $X$ to $Y$.
    
    \item {\bf Radon-Nikod\'ym operators}. 
    An operator $T:X \to Y$ between Banach spaces is called a \emph{Radon-Nikod\'ym operator} if $T$ takes {$L_1$}-bounded $X$-valued martingales into martingales which converge pointwise almost everywhere (see also Section~\ref{ss:martingales} below). 
    For equivalent definitions see~\cite{Edgar} and references therein.
    We will denote by $\RN$ the operator ideal of Radon-Nikod\'ym operators and by $\RN(X,Y)$ the set of all RN operators from $X$ to $Y$.

    \item {\bf Operators factoring through a Banach space with the Schur property}. 
    Partially following the notation of~\cite{Pietsch}, the ideals of these operators will be called $Op(SP)$.
    While it is clear that $Op(SP)\subset \DP$, the opposite inclusion does not hold, {see \cite[Proposition 3.1.8]{Pietsch}}.

	\item {\bf Operators factoring through a Banach space with the RNP}. 
     Partially following the notation of~\cite{Pietsch}, the ideals of these operators will be called $Op(RNP)$.
     Again we clearly have that $Op(RNP)\subset RN$. And again the opposite inclusion does not hold: an example is constructed in~\cite{GhoussoubJohnson1984} of a RN operator which does not factor through any space with the RNP.

    \item {\bf Operators not fixing any copy of $\ell_1$ (respectively, $L_1$)}. An operator $T:X\to Y$ does not fix any copy of $\ell_1$ (respectively, $L_1$) if there is no subspace $Z\subset X$ which is isomorphic to $\ell_1$ (respectively, $L_1$) and such that $T\restricted_Z$ is an isomorphic embedding. 
\end{itemize}

We will also briefly consider the class of strong Radon-Nikod\'ym operators (which does not form an ideal). An operator $T:X \to Y$ between Banach spaces is  a \emph{strong Radon-Nikodym operator} if $\overline{T(B_X)}$ is a RNP set. 
Recall that a closed bounded subset $C$ of a Banach space $X$ is called a \textit{Radon-Nikod\'ym set} (RNP set, in short) if for every probability space $(\Omega, \Sigma, \mathbb{P})$ and any vector measure $m : \Sigma \rightarrow X$ such that the set $\{ m(B) / \mathbb{P} (B) : B \in \Sigma, \, \mathbb{P}(B)>0\}$ is contained in $C$, there exists a Bochner-integrable Radon-Nikod\'ym derivative $f: \Omega \rightarrow X$ such that $m(B)= \int_B f \, dP$ for every $B \in \Sigma$.
It is known that if $T$ is strong Radon-Nikod\'ym, then it is Radon-Nikod\'ym (see \cite{Edgar}) while the converse is not true in general (considering any quotient map from $\ell_1$ onto $c_0$). Notice that a Banach space $X$ has the RNP if and only if the identity operator on $X$ is strong Radon-Nikod\'ym, 

Finally, for a finite measure space $(\Omega , \Sigma , \mu)$, a bounded operator $T : L_1(\mu) \to Y$ is said to be \emph{representable} if there exists $g \in L^{\infty}(\mu , Y)$ such that:
    $$\forall h \in L_1(\mu) , \quad T(h)= \int_\Omega hg \mathrm d \mu. $$

\subsection{Preliminaries on curve-flat Lipschitz functions}\label{ss:CF}

A \emph{Lipschitz curve} in a metric space $M$ is a Lipschitz map $\gamma\colon K\rightarrow M$ where $K$ is a compact subset of the real line. We warn the reader that our use of the term ``Lipschitz curve'' differs from a more traditional meaning where the domain of $\gamma$ would be required to be an interval. The next concept has been introduced in~\cite[Definition 5.21]{AGPP} for functions with values in $\Real$.
We say that a Lipschitz map $f\colon M\rightarrow N$ is \emph{curve-flat} if {for every compact subset $K$ of $\R$ and} every Lipschitz curve $\gamma\colon K\rightarrow M$, 
$$\lim_{\substack{y \to x \\ y\in K}}\frac{d(f(\gamma(x)),f(\gamma(y)))}{|x-y|} = 0$$
for $\lambda$-almost every $x \in K$, where $\lambda$ denotes the Lebesgue measure on $\RR$.
We will use the following notation 
$$\CF_0(M,N) := \{f \in \Lip_0(M,N) \; : \; f \text{ is curve-flat} \}.$$
If $N=\R$ we abbreviate the notation by letting $\CF_0(M):=\CF_0(M, \R)$. 
Let us mention that a routine argument shows that $\CF_0 (M, Y)$ is a closed subspace of $\Lip_0 (M,Y)$ when $Y$ is a Banach space.
We will see in Proposition~\ref{p:CFwsClosed} below that $\CF_0(M)$ is even weak$^*$-closed.

The curve-flat functions are close relatives of (but should not be confused with) \textit{uniformly locally flat} functions, that is functions $f$ such that
$$\lim\limits_{\substack{d(x,y) \to 0 \\ x \neq y}} \frac{d(f(x),f(y))}{d(x,y)} = 0.$$
We will use the following notation
\[
\lip_0(M,N):=\set{f \in \Lip_0(M,N): f \mbox{ is uniformly locally flat}}.
\]
If $N=\Real$ we let $\lip_0(M):=\lip_0(M,\Real)$.

We always have $\lip_0(M,N)\subset \CF_0(M,N)$ but equality might not hold. Indeed, consider   $M=\set{0}\cup\set{\frac1n:n \in \Natural}$ together with the metric $d$ inherited from $\Real$, then $d(\cdot,0)\in \CF_0(M)\setminus \lip_0(M)$. 

In this paper, we will handle curve-flat functions through equivalent descriptions based on classical results from metric geometry, namely Kirchheim's theorems \cite{Kirchheim}. For the convenience of the reader, let us state the relevant results to be used in the sequel. From now on, $\H^1$ stands for the Hausdorff 1-measure.

\begin{KirchThm}[Area Formula, {\cite[Corollary~8]{Kirchheim}}]
	Let $A$ be a Borel subset of $\RR$ and  $\gamma : A \to M$ be a Lipschitz map. Let the metric differential of $\gamma$ at $x \in A$ be the nonnegative number:
	$$MD(\gamma,x) := \lim_{\substack{y \to x \\ y \in A}} \dfrac{d(\gamma(x),\gamma(y))}{|x-y|}, \text{ whenever this limit exists.}$$
	 For each $y \in \gamma(A)$, $|\gamma^{-1}(y)| \in [1,\infty]$ denotes the cardinality of the preimage $\gamma^{-1}(y)$. 
	Then
	$$\int_A MD(\gamma,x)\, d\lambda(x) = \int_{\gamma(A)} |\gamma^{-1}(y)|\, d\H^1(y).$$
	Consequently, $MD(\gamma,x) = 0$ for $\lambda$-almost every $x \in A$ if and only if $\H^1(\gamma(A)) = 0$.
\end{KirchThm}

\begin{KirchLemma}
	For every metric space $(M, d)$ and every Lipschitz map $\gamma : A \to M$ with
	$A \subset \R$ Borel and $\H^1(\gamma(A)) > 0$, there exists a compact subset
	$S \subset A$ with $\lambda(S) > 0$ such that $\gamma \restricted_S$ is bi-Lipschitz.
\end{KirchLemma}

The previous lemma does not appear explicitly in but can be easily derived from \cite{Kirchheim}; for further details, see \cite[Lemma 1.11]{AGPP}. We can now present the desired equivalent formulations of the curve-flat property.

\begin{proposition} \label{prop:curveflat}
	Let $f:M \to N$ be a Lipschitz map between metric spaces. The following assertions are equivalent :
	\begin{enumerate}[$(i)$]
		\item $f$ is curve-flat;
		\item for every Lipschitz curve $\gamma\colon K\rightarrow M$, $\H^1(f\circ \gamma(K)) = 0$.
		\item for every bi-Lipschitz curve $\gamma\colon K\rightarrow M$, $\H^1(f\circ \gamma(K)) = 0$.
		\item for every Lipschitz curve $\gamma\colon K\rightarrow M$ and every Lipschitz map {$g : N \to \R $}, $\H^1(g \circ f\circ \gamma(K)) = 0$.
	\end{enumerate}
\end{proposition}

\begin{proof}
	The equivalence $(i) \iff (ii)$ follows from the Area formula. The implications $(ii) \implies (iii)$ and $(ii) \implies (iv)$ are obvious. Let us prove the contrapositive of $(iii) \implies (ii)$. Assume that there exists a Lipschitz curve $\gamma : K \to M$ such that $\H^1(f\circ \gamma(K))>0$.  By Kirchheim's lemma, there exists a compact subset $A$ of $K$ such that $\H^1(A) = \lambda(A)>0$ and $f\circ \gamma\restricted_A$ is bi-Lipschitz. Obviously $(\gamma\restricted_A)^{-1}=(f \circ \gamma\restricted_{A})^{-1}\circ f\restricted_{\gamma(A)}$ is Lipschitz.
 So $\gamma:A \to M$ is the desired bi-Lipschitz curve. The proof of $(iv) \implies (ii)$ is similar: Let $\gamma : K \to M$ a Lipschitz curve such that $\H^1(f\circ \gamma(K))>0$. Apply Kirchheim's lemma to obtain $A \subset K$ such that $\H^1(A) >0$ and $f\circ \gamma\restricted_A$ is bi-Lipschitz. Then define $g:= (f \circ \gamma \restricted_A)^{-1} : f\circ\gamma(A) \to A$ and extend $g$ to $N$ using McShane-Whitney's extension theorem ({see e.g. \cite[Theorem 1.33]{Weaver2}}). 
 We denote this extension by $\widetilde{g}$. Now observe that $\widetilde{g}$ is Lipschitz on $N$ and  such that $\H^1(\widetilde{g}\circ f \circ \gamma(A)) = \H^1(A)>0$, which is precisely the negation of $(iv)$.
\end{proof}  
  \begin{prop}\label{p:CFwsClosed}
    For any metric space $M$, the space $\CF_0(M)$ is weak$^*$-closed in $\Lip_0(M)$.
   \end{prop}   
   \begin{proof}
      By Krein-Smulian theorem, it is enough to show that $B_{\CF_0(M)}$ is w$^*$-closed.
      So let $f\in \overline{B_{\CF_0(M)}}^{w^*}$. 
      Clearly, $\norm{f}_L\leq 1$.
      Let us assume that $f\notin \CF_0(M)$.
      Then, by Proposition~\ref{prop:curveflat} there exists a Lipschitz curve $\gamma:K \to M$ such that $\H^1(f\circ\gamma(K))>0$. 
      Thus by Kirchheim's lemma we may as well assume that $\lambda(K)>0$ and $f\circ\gamma$ is bi-Lipschitz. 
      Let $C>0$ be the co-Lipschitz constant of $f\circ\gamma$, i.e., $C\abs{x-y}\leq | f\circ\gamma(x) - f\circ \gamma(y) |$ for every $x,y \in K$.
      We apply Lebesgue's density theorem to obtain a point $a \in K\setminus \set{\max K}$ where the density of $K$ is 1. 
      Now, given any $g \in B_{\CF_0(M)}$ and  $t\in K \cap (a,\max K]$, note that if $h$ is a McShane-Whitney's extension of $g \circ \gamma$ to $[a,\max K]$ then :
      \[
      \frac{\abs{g\circ \gamma(t)-g\circ \gamma(a)}}{\abs{t-a}} \leq \frac{1}{|t-a|}  \int_a^t | h' (s)| \, d\lambda(s) \leq 
      \norm{\gamma}_L\frac{\lambda([a,t]\setminus K)}{\abs{t-a}}.
      \]
      Since the density of $K$ at $a$ is 1, this last quantity goes to 0 as $t\to a+$. 
      In particular, 
      \[
      \set{\varphi \in \Lip_0(M):\frac{\abs{\duality{\delta(\gamma(t))-\delta(\gamma(a)),\varphi}}}{\abs{t-a}}>\frac{C}2}
      \]
      is a w$^*$-open neighborhood of $f$ which does not intersect with $B_{\CF_0(M)}$ if $t$ is chosen close enough to $a$.
      Thus $f\notin \overline{B_{\CF_0(M)}}^{w^*}$. 
      This contradiction shows that $B_{\CF_0(M)}$ is w$^*$-closed.
   \end{proof}

We conclude this section with the next central definition. 

\begin{definition}
	A metric space $M$ is said to be \textit{purely 1-unrectifiable} (p-1-u for short) if $\H^1(\gamma(A)) = 0$ for every
	$A \subset \R$ and every Lipschitz map $\gamma : A \to M$. 
\end{definition}
It is well known that, thanks to Kirchheim's lemma, $M$ is purely 1-unrectifiable if and only if $M$ contains no \emph{curve fragments}, that is, no image of a bi-Lipschitz embedding $ \gamma : K \to M$, where $K \subset \R$ is compact with $\lambda(K) > 0$. 
This, together with Proposition~\ref{prop:curveflat}, shows that $M$ is p-1-u if and only if $\Id_M \in \CF_0(M,M)$. We even get the following curious corollary.

\begin{cor}\label{cor:p-1-u_iff_CF_Sep_points_unif}
		If $M$ is any metric space such that $\CF_0(M)$ is w$^*$-dense in $\Lip_0(M)$, then $M$ is purely 1-unrectifiable.
\end{cor}
\begin{proof}
    We get by Proposition~\ref{p:CFwsClosed} and by the hypothesis that $\CF_0(M)=\Lip_0(M)$. Thus, applying  Proposition~\ref{prop:curveflat} (iv), we get $\Id_M \in CF_0(M,M)$, and so $M$ is p-1-u.
\end{proof}

{Studying further the isomorphic properties of $\Free(M)$ for $M$ complete p-1-u space is a promising research direction. 
Interesting results in this respect have been recently obtained by Basset in~\cite{Basset}. We also want to mention the following question, reminded to us by an anonymous referee (and stated, in a slightly different form, already in Godefroy's survey on Lipschitz free spaces {\cite[Problem 6.5]{GodefroySurvey}}):  It is not known, whether $\F(M)$ has the (metric) approximation property when $M$ is compact p-1-u. }

	\subsection{Preliminaries on martingales}\label{ss:martingales}
 
	Let $(\Omega,\A,\PP)$ be a probability space and let $X$ be a Banach space. We denote by $L_1(\Omega,\A,\PP; X)$ the space of (equivalence classes of) Bochner measurable functions $f : \Omega \to X$ such that $\int_{\Omega} \|f\|_X \, d\PP < \infty$ equipped as usual with the norm
	$$ \|f\|_{L_1(\Omega,\A,\PP; X)} = \E[\|f\|_X] = \int_{\Omega} \|f\|_X \, d\PP.$$
	In the sequel, we will suppress notation and write simply $L_1(X)$, or $L_1(\A;X)$ when the $\sigma$-algebra $\A$ needs to be emphasized.
	
	We recall that a sequence $(M_n)_{n=0}^\infty$ in $L_1(\A; X)$ is called a \emph{martingale} if there exists an increasing sequence $(\A_n)_{n=0}^\infty$ of $\sigma$-subalgebras of $\A$ (called a \emph{filtration}) such that for each $n \geq 0$, $M_n$ is $\A_n$-measurable and satisfies
	$$ M_n = \E^{\A_n}(M_{n+1}), $$
	where $\E^{\A_n}$ denotes the $X$-valued conditional expectation relative to $\A_n$ (see e.g. \cite[Section~1.2]{Pisier}). {Throughout this work, a bounded martingale will always refer to an $L_1(X)$-bounded martingale, meaning that $(\|M_n\|_X)_{n=0}^\infty$ is bounded in {$L_1(\mathcal{A}, \RR)$.}}
	We say moreover that $(M_n)_{n=0}^\infty$ is \emph{uniformly integrable} if the sequence of non-negative random variables $(\|M_n\|_X)_{n=0}^\infty$ is uniformly integrable. More precisely, this means that $(M_n)_{n=0}^\infty$ {is bounded} and that, for any $\eps>0$, there is a $\delta>0$ such that $$ \forall A \in \A, \quad \PP(A) < \delta \implies \sup_{n \geq 0} \int_A \|M_n\|_X \, d\PP < \eps.$$
	
	We will need the following characterization of Radon-Nikod\'ym operators.
	
	\begin{lemma}\label{Lemma:RN via unif integrable martingales}
		Let $X$ and $Y$ be Banach spaces. A bounded operator $T :X \to Y$ is Radon-Nikod\'ym if and only if for every uniformly integrable $X$-valued martingale $(M_n)_n$, $(T(M_n))_n$ converges in $L_1(Y)$.
	\end{lemma}
	{We only indicate the main ingredients of the proof {of Lemma \ref{Lemma:RN via unif integrable martingales}}. The proof of the ``if'' implication follows the exact same lines as the proof of \cite[Proposition~1.33]{Pisier}. {The ``only if'' direction} is an application of the Vitali Convergence Theorem  
(see \cite[Exercise 10 (b) on p. 133]{Rudin}).}

	\subsection{General preliminaries}
 Here we aim to state a general version of a well known principle that weakly null normalized sequences stay away from any finite dimensional subspace. 
 
\begin{lemma} \label{lem:liminf}
Let $X$ be a Banach space. Let $\sigma$ be a Hausdorff locally convex topology on $X$ such that $\norm{\cdot}$ is lower semi-continuous and let $(D_n)\subset X$ be a bounded sequence such that $D_n \rightarrow 0$ in $\sigma$. Then for every finite dimensional $Y \subseteq X$, we have that 
    \[
    \liminf_n \textup{dist} (D_n,Y)\geq \frac{\inf_n \|D_n\|}{2}. 
    \]
\end{lemma}

\begin{proof} Let $(y_{n_k}) \subset Y$ be such that $\|D_{n_k}-y_{n_k}\| \to \liminf_n \text{dist}(D_n,Y)$.
Moreover, by finite dimension of $Y$ and the fact that $\sigma$ is locally convex and Hausdorff we can assume that $y_{n_k} \to y$ in norm. 
It follows that  
$$\|D_{n_k}-y\| \to \liminf_n \text{dist} (D_n,Y).$$ 
 And so 
\begin{equation}
\label{eq:lem:liminf_1}
    \|y\|\leq \liminf_n \text{dist} (D_n,Y)
\end{equation}
by the lower semi-continuity of $\norm{\cdot}$ since $D_{n_k}-y \to -y$ in $\sigma$ by the hypothesis.
Also 
\begin{equation}
\label{eq:lem:liminf_2}
    \inf_n \|D_n\| -\|y\|\leq \liminf_k(\|D_{n_k}\|-\|y\|)\leq \lim_k\|D_{n_k}-y\|=\liminf_n \text{dist} (D_n,Y).
\end{equation}
Thus summing \eqref{eq:lem:liminf_1} and \eqref{eq:lem:liminf_2} gives 
$$
\frac{\inf_n \|D_n\|}{2} \leq \liminf_n \text{dist}(D_n,Y). 
$$
\end{proof}

It is well known that topologies which satisfy the hypothesis of the lemma are for instance the weak topology $w=\sigma(X,X^*)$ or, more generally, any topology $\sigma(X,S)$ where $S$ is subspace of $X^*$ which is norming for $X$, meaning that for all $x\in X$, $\|x\|_X=\sup\{x^*(x),\ x^*\in B_{X^*}\cap S\}$.

In particular, for a Banach space $X$, denote $S \subseteq L_\infty(\Omega, \mathcal{A}, \mathbb{P},X^*)$ the normed linear space of $X^*$-valued simple functions, i.e. measurable functions with only finitely many distinct values. This space is easily seen to embed isometrically into $(L_1(\mathcal{A},X))^*$ (see \cite[p. 97]{Diestel}).

\begin{cor} \label{Cor:liminf}
Let $X$ be a Banach space and let $(D_n)\subset L_1(X)$ be a bounded sequence such that $D_n \rightarrow 0$ in $\sigma(L_1 (X), S)$. Then for every finite dimensional $Y \subseteq L_1(X)$, we have that 
    \[
    \liminf_n \textup{dist} (D_n,Y)\geq \frac{\inf_n \|D_n\|}{2}. 
    \]
\end{cor}

\begin{proof} 
It follows from the definition of the Bochner-Lebesgue space that the space $S$ of simple functions is norming for $L_1(X)$. 
As we already mentioned it can be easily deduced that $\|\cdot\|$ is $\sigma(L_1(X),S)$-lower semi-continuous.
We can thus apply Lemma~\ref{lem:liminf} to get the desired conclusion.
\end{proof}
	
	\section{Dunford-Pettis Lipschitz operators}
	\label{section:DP}

 In this section we prove the first ``half'' of Theorem~\ref{t:flagship}.
	
	\begin{thm}\label{t:DP}
		{Let $M, N$ be pointed metric spaces, with $M$ complete,} and let $f \in \Lip_0(M,N)$. 
		Then the following properties are equivalent:
		\begin{enumerate}
			\item $f$ is curve-flat,
			\item $\widehat{f}$ is Dunford-Pettis,
			\item $\widehat{f}$ does not fix any copy of $L_1$.
		\end{enumerate}
	\end{thm}
    
		{It has been proved in~\cite{Bourgain1980IJM} that if $T:L_1 \to L_1$ is not Dunford-Pettis, then it fixes a copy of $\ell_1(\ell_2)$. The equivalence of $(2)$ and $(3)$ in Theorem \ref{t:DP} shows that linearizations of Lipschitz maps which are not Dunford-Pettis satisfy a stronger property.}

	The majority of this section is dedicated to showing that (1) implies (2) in Theorem \ref{t:DP}. Indeed, as we will later see in its proof, (2) implies (3) is direct since $L_1$ contains a normalized weakly null sequence, while (3) implies (1) follows from Kircheim's Lemma and Godard's work on Lipschitz-free spaces over subsets of the real line.

	This section is divided into three subsections. In the first one we address (1) implies (2) of Theorem \ref{t:DP} in the case where $M$ is any compact metric space, leveraging a profound result established by D. Bate \cite{Bat20}. In the second, we deal with the general case thanks to a compact reduction principle along the lines of \cite{APP21}. In the last subsection we use both previous results to finally prove Theorem \ref{t:DP}.

	\subsection{\texorpdfstring{$f$}{f} is curve-flat implies \texorpdfstring{$\widehat{f}$}{f hat} is Dunford-Pettis: The compact case}

	As previously mentioned, a key ingredient in this setting is a result due to D. Bate \cite{Bat20}. To ensure clarity and ease of reference, we will present this result in its entirety. First, let us introduce some notation.
	
	\begin{definition}
		Given $\varepsilon>0$ and $\theta>0$, we say that a Lipschitz map $f\colon M\rightarrow N$ is \emph{$\varepsilon$-flat at radius $\theta$} if $d(f(x),f(y))\leq \varepsilon d(x,y)$ for all $x,y\in M$ with $d(x,y)\leq \theta$. 
		It is \emph{$\varepsilon$-flat} if it is $\varepsilon$-flat at radius $\theta$ for some $\theta>0$. Denote by $\text{lip}_0^{\varepsilon}(M)$ the set of $\varepsilon$-flat Lipschitz functions in $\text{Lip}_0(M)$. Note that if $\varepsilon_1<\varepsilon_2$, then $\text{lip}_0^{\varepsilon_1}(M)\subset \text{lip}_0^{\varepsilon_2}(M)$.	
	\end{definition}

	In the next lemma, if $t \in K \subset \R$ and $\gamma : K \to M$ is Lipschitz, then $\Lip(\gamma,t)$ stands for the \textit{pointwise Lipschitz constant} of $\gamma$ at $t$, defined as:
	$$ \Lip(\gamma,t):=\limsup_{x \to t} \dfrac{d(\gamma(t) , \gamma(x))}{|t-x|}.$$
	
	\begin{lem}[{\cite[Lemma~3.4]{Bat20}}] \label{Lemma_3.4_Bate}
		Let $M$ be a compact metric space, let $X$ be a Banach space which contains $M$ isometrically. Let $\delta>0$ and let $x^*\in X^*$. Suppose that every bi-Lipschitz curve $\gamma\colon K\rightarrow X$ satisfying 
		$$(x^*\circ \gamma)'(t)\geq \delta\|x^*\|\Lip(\gamma,t),\qquad \mathcal{H}^1\text{-a.e. in }t\in K $$
		also satisfies $\mathcal{H}^1(\gamma(K)\cap M)=0$.

		Then, for every $\varepsilon>0$, there exists a Lipschitz function $g\colon X\rightarrow \mathbb{R}$ such that:
		\begin{itemize}
			\item $\frac{|g(x)-g(y)|}{\|x-y\|}\leq \frac{|x^*(x)-x^*(y)|}{\|x-y\|}+3\|x^*\|\delta$ for every $x\neq y\in X$. In particular, $\|g\|_{L}\leq \|x^*\|(1+3\delta)$,
			\item $|x^*(x)-g(x)|<\varepsilon$ for all $x\in X$, and
			\item there exists $\theta>0$ such that $g\restricted_{M}$ is $3\delta{\|x^*\|}$-flat at radius $\theta$.
		\end{itemize}
	\end{lem}

Clearly, if $0 \in M \subset X$, the function $g$ from the above lemma can be also assumed to satisfy $g(0)=0$. 
This is how we will use it below.
The next lemma and its proof are essentially contained in Section~3 of \cite{Bat20}.  We present a direct proof here for clarity and convenience.

	\begin{lem}
		\label{Lemma_Bate}
		If $M$ is compact then $B_{\CF_0(M)}\subset \overline{B_{\lip_0^\tau(M)}}^{w^*}$  for every $\tau>0$.
	\end{lem}
	
	\begin{proof}
        	Let $f\in \CF_0(M)$ with $\|f\|_{L}\leq 1$, and let $\tau>0$. Since the weak$^*$ topology coincides with the topology of pointwise convergence on norm-bounded subsets of $\text{Lip}_0(M)$, {it suffices to show that for every $\eta>0$, there exists a function $g\in \text{lip}_0^{\tau} (M)$ with $\|g\|_\text{Lip}\leq 1$ such that $|f(x)-g(x)|<2\eta$ for all $x\in M$}.

		We start by showing that the linear extension of $f$, i.e. the map $\overline{f}\colon \mathcal{F}(M)\rightarrow \mathbb{R}$, satisfies the hypothesis of Lemma \ref{Lemma_3.4_Bate} for every $\delta>0$. Suppose that $\gamma\colon K\rightarrow \F(M)$ is a bi-Lipschitz curve satisfying 
		$$(\overline{f} \circ \gamma)'(t)\geq \delta\|\overline{f}\|\Lip(\gamma,t),\qquad \mathcal{H}^1\text{-a.e. in }t\in K. $$
		We may of course assume that $f\neq 0$, so that the right-hand side is everywhere strictly positive. Now, arguing by contradiction, suppose that $\mathcal{H}^1(\gamma(K)\cap M)>0$. Clearly $A := \gamma^{-1}(\gamma(K)\cap M)\subset K$ has positive measure, and the curve $\gamma\restricted_{A}\colon A\rightarrow M$ is bi-Lipschitz with values in $M$. 
		Since $f$ is curve-flat, we obtain that $(f\circ \gamma\restricted_{A})'(t)=0$ almost everywhere in $A$, which is a contradiction.

		Now, we can apply Lemma~\ref{Lemma_3.4_Bate} to $\overline{f}$, with {$\varepsilon=\eta$ and $\delta=\delta(\tau,\eta)>0$} to be determined later, to obtain $\tilde{g}\in \text{lip}_0^{3\delta}(M)$ with $\|\tilde{g}\|_L\leq (1+3\delta)$ and such that {$|\tilde{g}(x)-f(x)|<\eta$} for all $x\in M$. We only need to normalize $\tilde{g}$ to get the result. Indeed, consider $g=\frac{1}{1+3\delta}\tilde{g}$. Then $g$ also belongs to $\text{lip}_0^{3\delta}(M)$ and now $\|g\|_L\leq 1$. Finally, we have for any $x\in M$:
		{\begin{align*}
			|g(x)-f(x)| &\leq |\tilde{g}(x)-f(x)|+\Big(1-\frac{1}{1+3\delta}\Big)|\tilde{g}(x)| \\
			&\leq \eta+\Big(1-\frac{1}{1+3\delta}\Big)(1+3\delta)\text{rad}(M) \\
            &= \eta + 3\delta\text{rad}(M).
		\end{align*}}
		Consequently, for {$\delta = \min\left\{\frac{\tau}{3},\frac{\eta}{3\text{rad}(M)}\right\}$}, we obtained the desired conclusion. 
	\end{proof}
	
	We can now prove the main result of this subsection. We will employ a gliding hump type argument inspired by Kalton's proof that the Lipschitz-free space over a snowflaked metric space has the Schur property 
 (\cite[Theorem 4.6]{Kalton04}). For the next proof, recall the notation $[S]_{\delta} := \{x \in M : d(x,S) \leq \delta\}$ for $S$ a non empty subset of $M$ and $\delta>0$.
	
	\begin{prop} \label{prop:CFimpliesDPCompactCase}
		Let $M$ and $N$ be metric spaces with $M$ compact, and let $f \in \Lip_0(M,N)$. If $f$ is curve-flat, then $\widehat{f}\colon \mathcal{F}(M)\rightarrow \mathcal{F}(N)$ is Dunford-Pettis.
	\end{prop}
	
	\begin{proof}
	Suppose that $f\colon M\rightarrow N$ is a curve-flat Lipschitz map with $\|f\|_\text{Lip}=1$. We will prove by contradiction that $\widehat{f}\colon \mathcal{F}(M)\rightarrow\mathcal{F}(N)$ is Dunford-Pettis. Suppose there exists a weakly null sequence $(\gamma_k)_k\subset \mathcal{F}(M)$ such that $\|\widehat{f}(\gamma_k)\|= 1$ for all $k\in\mathbb{N}$. We may assume without loss of generality that the support of $\gamma_k$ is finite for all $k\in\mathbb{N}$. Choose a decreasing sequence of positive numbers $(\varepsilon_n)_n$ such that $\sum_{n} \varepsilon_n<\frac{1}{4}$. We will construct, by induction, a bounded sequence of functions $(\varphi_n)_n\subset \text{Lip}_0(M)$ and a subsequence $(\gamma_{k_n})_n$ such that:
		
		\begin{enumerate}
			\item $\varphi_n$ is $\left(\sum_{i=1}^n\varepsilon_i\right)$-flat at some positive radius $0<\theta_n<1$,
			\item $\|\varphi_n\|_{L}\leq 1+\sum_{i=1}^n \varepsilon_i$, and
			\item $\langle \varphi_n,\gamma_{k_j}\rangle \geq \frac{1}{2}-\sum_{i=1}^n\varepsilon_i$ for all $j=1,\dots,n$.    
		\end{enumerate}
		
		For the initial step, let $\gamma_{k_1}=\gamma_1$. Since $\|\widehat{f}(\gamma_1)\|=1$, there exists a Lipschitz function $g_1\in {B_{\Lip_0(N)}}$ such that $\langle g_1,\widehat{f}(\gamma_1)\rangle =1 = \langle g_1\circ f,\gamma_1\rangle$. The function $g_1\circ f$ is curve-flat, so by Lemma \ref{Lemma_Bate} there exists $\varphi_1\in B_{\text{Lip}_0(M)}$ such that $\langle \varphi_1,\gamma_1\rangle \geq 1-\varepsilon_1$ and such that $\varphi_1$ is $\varepsilon_1$-flat at some radius $\theta_1\in(0,1)$. 
		
		Let us fix $n\in\mathbb{N}$ and suppose we have constructed $\gamma_{k_j}$ and $\varphi_j$ for every $j\leq n$. Fix $\tau>0$ to be determined later. By compactness of $M$, there exists a finite set $E_{n+1}$, containing the support of $\gamma_{k_j}$ for all $j=1\leq n$, and such that $M=[E_{n+1}]_\tau$.
		
		Since $f(E_{n+1})$ is finite as well, $\mathcal{F}(f(E_{n+1}))$ is finite-dimensional. Given that $(\widehat{f}(\gamma_k))_k$ is a normalized weakly null sequence, Lemma~\ref{lem:liminf} gives that \[\liminf_k \text{dist}(\widehat{f}(\gamma_k),\mathcal{F}(f(E_{n+1})))\geq\frac{1}{2}. \]
  \sloppy Therefore, combining with the Hahn-Banach separation theorem, we can find $k_{n+1}>k_n$ and $g_{n+1}\in
		\text{Lip}_0(N)$ with $\|g_{n+1}\|_L\leq 1$, $g_{n+1}(f(E_{n+1}))=\{0\}$ and $\langle g_{n+1},\widehat{f}(\gamma_{k_{n+1}})\rangle>\frac{1}{2}-\tau$. Since $(\gamma_k)_k$ is weakly null, we may assume as well that $\langle \varphi_n,\gamma_{k_{n+1}}\rangle <\tau$.

		The map $g_{n+1}\circ f\in \text{Lip}_0(M)$ is curve-flat, and satisfies 
		$$ (g_{n+1}\circ f) (u)=0,~\text{ for all }u\in E_{n+1},~\text{ and } \langle g_{n+1}\circ f, \gamma_{k_{n+1}}\rangle >\frac{1}{2}-\tau.$$
		Therefore, thanks to Lemma \ref{Lemma_Bate}, it can be approximated by a function $h_{n+1}\in \text{Lip}_0(M)$ with $\|h_{n+1}\|_L\leq 1$ and such that: 
		\begin{itemize}[leftmargin=*]
			\item $|h_{n+1} (u)|<2\tau$, for all $u\in M$,
			\item $\langle h_{n+1}, \gamma_{k_{n+1}}\rangle >\frac{1}{2}-2\tau$,
			\item $\langle h_{n+1},\gamma_{k_j}\rangle <\tau $ for all $j=1,\dots,n $, and
			\item $h_{n+1}$ is $\varepsilon_{n+1}$-flat at some radius $\theta_{n+1}$, which can be chosen so that $0<\theta_{n+1} \leq \theta_n$.
		\end{itemize}
		
		Now, define $\varphi_{n+1}=\varphi_n+h_{n+1}$. Let us check that $\varphi_{n+1}$ satisfies conditions (1), (2) and (3). First, condition (1) is clear since $\varphi_n$ is $(\sum_{i=1}^n\varepsilon_i)$-flat and $h_{n+1}$ is $\varepsilon_{n+1}$-flat. Next we estimate its Lipschitz constant. Let $x,y\in M$. If $d(x,y)<\theta_{n}$, then, since $\varphi_n$ is $\left(\sum_{i=1}^n\varepsilon_i\right)$-flat at radius $\theta_n$:
		\begin{align*}
			|\varphi_{n+1}(x)-\varphi_{n+1}(y)|&\leq |\varphi_n(x)-\varphi_n(y)|+|h_{n+1}(x)-h_{n+1}(y)|\\
			&\leq \Big(\sum_{i=1}^n\varepsilon_i+1\Big)d(x,y)	
		\end{align*}
		Otherwise, if $d(x,y) \geq \theta_{n}$ then:
		\begin{align*}
			|\varphi_{n+1}(x)-\varphi_{n+1}(y)|&\leq |\varphi_n(x)-\varphi_n(y)|+|h_{n+1}(x)|+|h_{n+1}(y)|\\
			&\leq \Big(1+\sum_{i=1}^n\varepsilon_i\Big)d(x,y)+4\frac{\tau}{\theta_n}d(x,y)\\
			&=\Big(1+\sum_{i=1}^n \varepsilon_i+4\frac{\tau}{\theta_n}\Big)d(x,y).
		\end{align*}
		Therefore, we conclude that $\|\varphi_{n+1}\|_L\leq\Big(1+\sum_{i=1}^n \varepsilon_i+4\frac{\tau}{\theta_n}\Big) $.
		
		Finally, let us examine the estimates in (3). If $j\leq n$:
		\begin{align*}
			\langle\varphi_{n+1},\gamma_{k_j}\rangle\geq\langle \varphi_{n},\gamma_{k_j}\rangle -|\langle h_{n+1},\gamma_{k_j}\rangle|>\frac{1}{2}-\sum_{i=1}^n\varepsilon_i-\tau.
		\end{align*}
		Simultaneously, we have:
		\begin{align*}
			\langle \varphi_{n+1},\gamma_{k_{n+1}}\rangle \geq \langle h_{n+1},\gamma_{k_{n+1}}\rangle -|\langle \varphi_n,\gamma_{k_{n+1}} \rangle|>\frac{1}{2}-3\tau.
		\end{align*}
		Choosing $\tau$ so that $4\frac{\tau}{\theta_n}<\varepsilon_{n+1}$ yields condition (2). In particular, this implies $3\tau<\varepsilon_{n+1}$, thereby proving (3) as well. The induction process is now complete.
		
		To conclude, passing to a subsequence if necessary, we may assume that $(\varphi_{n})_n$ converges to a Lipschitz map $\tilde{\varphi}\in \text{Lip}_0(M)$ in the weak$^*$ topology. For a fixed $j\in\mathbb{N}$, we have that $\langle \varphi_n,\gamma_{k_j}\rangle>\frac{1}{2}-\sum_{k=1}^\infty\varepsilon_k>\frac{1}{4}$ whenever $n\geq j$. Therefore, $\langle \tilde{\varphi},\gamma_{k_j}\rangle>\frac{1}{4}$ for all $j\in\mathbb{N}$. This is a contradiction with the fact that $(\gamma_{k_j})_j$ is weakly null.
	\end{proof}

	\subsection{Compact determination of Dunford-Pettis operators}
	
	Once we have established that (1) implies (2) in Theorem \ref{t:DP} for compact metric spaces, we proceed to extend this result to general metric spaces thanks to the following compact reduction result:

	\begin{proposition}
		\label{prop:DPisCompactlyDetermined}
		{Let $M, N$ be pointed metric spaces, with $M$ complete,} and let $f\in \Lip_0(M,N)$. The property ``$\widehat{f}$ is Dunford-Pettis'' is compactly determined, that is, $\widehat{f}$ is Dunford-Pettis if and only if for every compact $K \subset M$ such that $0\in K$, $\widehat{f\restricted_K}$ is Dunford-Pettis. 
	\end{proposition}
	
	\begin{proof}
		One direction is clear: If $\widehat{f}$ is Dunford-Pettis, then so is  $\widehat{f\restricted_K}$ for every $K\subset M$ compact containing $0$. The converse will follow from the compact reduction principle described in \cite{APP21}. Suppose that $f : M \to N$ is a Lipschitz map such that $\widehat{f\restricted_K}$ is Dunford-Pettis for every $K\subset M$ compact containing $0$. Let $(\gamma_n)_n \subset \F(M)$ be weakly null. We will prove that $(\widehat{f}(\gamma_n))_n$ is norm-null. Let $W=\{\gamma_n : n \in \N\}$ and fix $\ep>0$. By \cite[Theorem~2.3]{APP21}, there exist a compact $K\subset M$ containing $0$ and a  linear mapping $T:\mathrm{span}(W) \to \F(K)$ such that
		\begin{itemize}
			\item $\norm{\mu-T\mu}\leq \varepsilon$ for all $\mu\in W$, and
			\item  there is a sequence of bounded linear operators $T_k:\Free(M)\rightarrow \Free(M)$ such that $T_k\rightarrow T$ uniformly on $W$.
		\end{itemize}
		A classical exchange of limits argument (see \cite[Corollary~2.4]{APP21} for more details) yields that the sequence $(T(\gamma_n))_n$ is weakly null in $\F(K)$. Since $\widehat{f\restricted_K}$ is Dunford-Pettis, $\|\widehat{f}(T\gamma_n)\| \to 0$. So there exists $n_0 \in \N$ such that for every $n \geq n_0$, $\|\widehat{f}(T\gamma_n)\| \leq \ep$. By the triangle inequality we obtain: 
		$$\forall n \geq n_0, \quad \|\widehat{f}(\gamma_n)\| \leq \|\widehat{f} \, \| \|\gamma_n - T\gamma_n\| + \| \widehat{f}(T\gamma_n) \| \leq (\|f\|_L+1)\ep. $$
		Since $\ep>0$ was arbitrary, this proves that $(\widehat{f}(\gamma_n))_n$ is norm-null, and therefore $\widehat{f}$ is Dunford-Pettis.
	\end{proof}
	
	We now have all the ingredients to prove the main result of the section.
	\subsection{Proof of Theorem \ref{t:DP}}
	
	\begin{proof}[Proof of Theorem \ref{t:DP}]
		The implication $(1) \implies (2)$ follows from Propositions \ref{prop:CFimpliesDPCompactCase} and \ref{prop:DPisCompactlyDetermined}, while $(2) \implies (3)$ is direct since the unit sphere of $L_1$ contains a weakly null sequence. 
		
		To show that $(3) \implies (1)$, we proceed by contrapositive. Thus, suppose that $f$ is not curve-flat. By Proposition~\ref{prop:curveflat} there exists $K\subset \R$ compact, $\gamma : K \to M$ bi-Lipschitz such that $\H^1(f\circ \gamma(K))> 0$. By Kirchheim's lemma, there exists $A$ compact, $A \subset K$ such that $\lambda(A)>0$ and $f\circ \gamma \restricted_A$ is bi-Lipschitz. Now, thanks to Godard's work \cite{Godard}, $\F(A)$ is isomorphic to $L_1$. As $\gamma\restricted_A$ and $f\circ \gamma \restricted_A$ are bi-Lipschitz, $\widehat{\gamma\restricted_A}$ and $\widehat{f\circ \gamma \restricted_A}$ are isomorphisms from $\F(A)$ onto $\F(\gamma(A))$ and $\F(f\circ \gamma (A))$, respectively. Since $\widehat{f \circ \gamma}=\widehat{f}\circ \widehat{\gamma}$, it follows that $\widehat{f}\restricted_{\F(\gamma(A))}$ is an isomorphism from $\F(\gamma(A))$ onto $\F(f\circ \gamma (A))$. This proves that $\hat{f}$ fixes a copy of $L_1$.
	\end{proof}

	\section{Radon-Nikod\'ym Lipschitz operators}\label{section:RN}
 In this section, we prove the second ``half'' of Theorem~\ref{t:flagship}.

	\begin{thm}\label{t:RN}
		{Let $M, N$ be pointed metric spaces, with $M$ complete,} and let $f \in \Lip_0(M,N)$. Then the following properties are equivalent:
		\begin{enumerate}
			\item $f$ is curve-flat,
			\item $\widehat{f}$ is Radon-Nikod\'ym,
			\item $\widehat{f}$ does not fix any copy of $L_1$.
		\end{enumerate}
	\end{thm}
	Note that the proof of (2) $\implies$ (3) is trivial as there exists a $L_1$-bounded martingale in $L_1(L_1)$ which does not converge almost everywhere. On the other hand, (3) $\implies$ (1) is shown in Theorem \ref{t:DP} already. Hence, it only remains to deal with the implication (1) $\implies$ (2), which will be the focus of the present section. 
	
	The strategy to show (1) $\implies$ (2) is the same as the one we followed in Section~\ref{section:DP}: first we show that it holds in compact metric spaces, and then we prove that the property ``$\widehat{f}$ is Radon-Nikod\'ym" is compactly determined. In both results we will be dealing with Banach space valued martingales and quasi-martingales. 
 
Recall that a stochastic process $(S_n)_n$ in $L_1 (\mathcal{A}, X)$ adapted to a filtration $(\mathcal{A}_n)$ is said to be a \emph{quasi-martingale} if 
\[
\sum_{n=1}^\infty \| \mathbb{E}^{\mathcal{A}_{n-1}} ( S_n - S_{n-1}) \| < \infty. 
\]
We refer to \cite[Remark~2.16]{Pisier} for more details about quasi-martingales.

\begin{lemma}\label{l:WeakConvergenceOfMartingales}
Let $X$ be a Banach space, $(M_n)$ be a bounded and uniformly integrable quasi-martinagle in $L_1 (X)$ and denote $D_n := M_n - M_{n-1}$ (with $M_0 :=0$). Then $D_n \to 0$ in $\sigma(L_1(X),S)$ where $S$ denotes the space of simple functions in $L_\infty(\Omega, X^*) \subseteq L_1 (\Omega, X)^*$.
\end{lemma}

\begin{proof} 
    Let $g \in S$ be fixed. Note that the set $\{g(\omega) : \omega \in \Omega\}$ is a finite set; thus generates a finite dimensional subspace $E$ of $X^*$. Let us fix a basis $\{x_1^*,\ldots,x_k^*\}$ of $E$. Due to finite dimensionality, there exists $C>0$ such that $\max_i |\lambda_i| \leq C \|x^*\|$ for every $x^* = \sum_{i=1}^k \lambda_i x_i^* \in E$. If $\omega \in \Omega$ then we can write $g(\omega) = \sum_{i=1}^k \lambda_i(\omega)x_i^*$ for some $\lambda_i (\omega) \in \mathbb{R}$, $i=1,\ldots, k$. Notice that $(x_i^* \circ M_n)_n$ is a scalar valued, bounded, and uniformly integrable quasi-martingale for each $i=1,\ldots, k$. Therefore it converges in $L_1$ (see \cite[Remark 2.16]{Pisier}). This readily implies that each $(x_i^* \circ D_n)_n$ goes to $0$  in $L_1$ as $n\rightarrow \infty$. We conclude with the following estimates:
	\begin{align*}
	| \<g , D_n \>|  &\leq  \int_\Omega | \< g(w) , D_n(\omega) \>| \,  d\PP(\omega) \\
	&\leq \int_\Omega  \sum_{i=1}^{k} |\lambda_i(\omega)| \cdot | x_i^*\circ D_n(\omega) | \,  d\PP(\omega) \ \\
	&\leq C \|g\|_\infty \sum_{i=1}^{k} \int_\Omega | x_i^*\circ D_n(\omega) | \,  d\PP(\omega) \rightarrow 0. \qedhere
	\end{align*}
\end{proof}
	
	\subsection{Proof of the compact case of Theorem~\ref{t:RN}}
	
    Here we will deal with the compact case. Up to the use of martingales, the proof will largely resemble that of Proposition \ref{prop:CFimpliesDPCompactCase} which established the implication $``f\in \CF_0" \implies ``\hat{f} \in \DP"$.

	\begin{prop}
		\label{prop:CFimpliesRNCompactCase}
		Let $M$ and $N$ be metric spaces with $M$ compact, and let $f \in \Lip_0(M,N)$. If $f$ is curve-flat, then $\widehat{f}\colon \mathcal{F}(M)\rightarrow \mathcal{F}(N)$ is Radon-Nikod\'ym.
	\end{prop}

	\begin{proof}
  We assume that there is a uniformly integrable martingale $(M_n) \subset L_1(\Free(M))$ such that $(\widehat{f} \circ M_n)_n$ does not converge in $L_1(\Free(N))$.

		Let $D_{n}=M_{n+1}-M_n$. Up to passing to a subsequence and multiplying by a positive constant, we may assume that $\|\widehat{f} \circ D_n\| \geq 1$ for every $n \in \N$.
		We will employ a gliding hump argument to show that $(M_n)$ is not bounded. 
		
	First, using a standard approximation technique, we may assume that $(M_n)_{n=0}^\infty$ is adapted to a filtration $(\A_n)_{n=0}^\infty$ where each $\A_n$ is finite (see the proof of \cite[Proposition~4.3]{AGPP} for more details). Next, by density of the finitely supported elements in $\lipfree{M}$, we may replace $(M_n)_n$ with a stochastic process $(\cl{M}_n)_n$ adapted to $(\A_n)_{n=0}^\infty$ satisfying
		\begin{itemize}[itemsep=3pt]
            \item {$\lim_{n \to \infty}\|\cl{M}_n-M_n\|_{L_1}=0$}
			\item $\cl{M}_n$ is constant on each atom of $\A_n$ and not just essentially constant,
			\item $\cl{M}_n(\omega)$ is finitely supported for every $n \in \NN$ and $\omega \in \Omega$,
			\item $(\cl{M}_n)_n$ is $L_1(\lipfree{M})$-bounded,
			\item $\|\widehat{f} \circ \cl{D}_n \| \geq \frac{1}{2}$, where $\cl{D}_n:=\cl{M}_{n+1} - \cl{M}_n$ for every $n \in \N$,
			\item for every $n \leq i$, $\|\E^{\A_n}(\cl{M}_{i})-\cl{M}_n\|_{L_1} \leq 2^{-n}$. 
        \end{itemize}
		
		Notice that $(\cl{M}_n)_n$ is still uniformly integrable and the last condition implies that $(\cl{M_n})_{n=0}^\infty$ is a quasi-martingale. 
		By Lemma \ref{l:WeakConvergenceOfMartingales}, we have that the sequence $(\cl{D_n})_n$ is $\sigma(L_1(\F(M)) , S)$-null, where $S$ denotes the space of simple functions in $L_\infty(\Omega, \F(M)^*) \subseteq L_1 (\Omega, \F(M))^*$.

		\smallskip
		
		Now choose a decreasing sequence of positive numbers $(\varepsilon_n)_n$ in such a way that $\sum_{n} \varepsilon_n<\frac{1}{8}$. We will construct, by induction, a bounded sequence $(\varphi_n)_n$ and a subsequence $(\overline{D}_{k_n})_n$ such that:
		
		\begin{enumerate}
			\item $\varphi_n \in L_1(\mathcal{A}_{k_n+1} , \F(M))^*  = L_\infty(\mathcal{A}_{k_n+1} , \Lip_0(M))$,
			\item For every $\omega \in \Omega$, $\varphi_n(\omega)$ is $\left(\sum_{i=1}^n\varepsilon_i\right)$-flat at some positive radius $0<\theta_n<1$,
			\item $\|\varphi_n\| \leq 1+\sum_{i=1}^n \varepsilon_i$, and
			\item $\langle \varphi_n,\cl{D}_{k_i}\rangle \geq \frac{1}{4}-\sum_{i=1}^n\varepsilon_i$ for all $i=1,\dots,n$.
		\end{enumerate}
		
		For the initial step, let $\cl{D}_{k_1}=\cl{D}_1$. Since $\|\widehat{f} \circ \cl{D}_1 \|\geq \frac{1}{2}$, there exists a map $g_1\in L_\infty(\A_2 , \Lip_0(N))$ with $\|g_1\|=1$ such that $\langle g_1,\widehat{f} \circ \cl{D}_1\rangle \geq \frac{1}{2}$. Notice that 
		$$ \langle g_1, \widehat{f} \circ \cl{D}_1 \rangle = \int_\Omega \langle g_1(\omega),\widehat{f}(\cl{D}_1(\omega))\rangle d\PP = \int_\Omega \langle g_1(\omega)\circ f,\cl{D}_1(\omega)\rangle d\PP.  $$
		For each $\omega \in \Omega$, the function $g_1(\omega)\circ f$ is curve-flat. So by Lemma \ref{Lemma_Bate}, there exists a simple function $\varphi_1 \in L_\infty(\A_2 , \Lip_0(M))$ such that $\|\varphi_1\| \leq 1$, $\langle \varphi_1,\cl{D}_1\rangle \geq \frac{1}{2}-\varepsilon_1$ and finally $\varphi_1(\omega)$ is $\varepsilon_1$-flat at some radius $\theta_1\in(0,1)$ for every $\omega \in \Omega$. 
  
		Let us fix $n\in\mathbb{N}$ and suppose we have constructed $\cl{D}_{k_i}$ and $\varphi_i$ for every $i\leq n$. Fix $0< \tau < \frac{1}{16}$ to be determined later. By compactness, there exists a finite set $E_{n+1}$ containing $$\bigcup_{i=1}^n \bigcup_{\omega \in \Omega}\left(\supp(\cl{M}_{k_i}(\omega))\cup \supp(\cl{M}_{k_i+1}(\omega))\right)$$
  and such that $M=[E_{n+1}]_\tau$.

		Since $f(E_{n+1})$ is finite as well, $\mathcal{F}(f(E_{n+1}))$ is finite-dimensional. 
		Notice that for each $g \in L_1 (\A_{k_n+1} , \mathcal{F}( N ) )^* = L_\infty (\mathcal{A}_{k_n +1}, \Lip_0 (N))$, we have
			\[
			\langle  g,  \widehat{f}  \circ \cl{D}_k  \rangle = \int_{\Omega}  \langle g (\omega) \circ f, \cl{D}_{k} (\omega)\rangle d\PP \rightarrow 0,
			\]
			where the convergence is guaranteed thanks to Lemma
			\ref{l:WeakConvergenceOfMartingales}. Thus, applying Corollary \ref{Cor:liminf} to $(\widehat{f} \circ \cl{D}_k)_k$, $\mathcal{A}_{k_n +1}$ and $\mathcal{F}(f(E_{n+1}))$, we observe that 
			\[
			\liminf_k \text{dist}\big(\widehat{f} \circ \cl{D}_k,L_1 (\A_{k_n+1} , \mathcal{F}(f(E_{n+1}))){\big)}\geq\frac{\inf_k \|\widehat{f}  \circ \cl{D}_k \| }{2} \geq \frac{1}{4}. 
			\] 
			So there exists $k_{n+1} > k_n$ such that $\text{dist}\big(\widehat{f} \circ \cl{D}_{k_{n+1}} ,L_1(\A_{k_n+1} , \mathcal{F}(f(E_{n+1}))){\big)}>\frac{1}{4} - \tau$. Since $\widehat{f} \circ \cl{D}_{k_{n+1}} \in L_1(\A_{k_{n+1}+1}, \F(N))$, by Hahn Banach separation theorem, there exists $g_{n+1}\in
			L_\infty(\A_{k_{n+1}+1} , \text{Lip}_0(N))$ with $\|g_{n+1}\|\leq 1$ such that  \begin{itemize}[leftmargin=*]
				\item $\langle g_{n+1}, L_1(\A_{k_n+1} , \mathcal{F}(f(E_{n+1})) ) \rangle = \{0\}$, and 
				\item $\langle g_{n+1},\widehat{f} \circ \cl{D}_{k_{n+1}} \rangle>\frac{1}{4}-\tau$.
			\end{itemize}
		Knowing $(\cl{D}_k)_k$ is $\sigma(L_1(\F(M)) , S)$-null, we may assume that $| \langle \varphi_n,\cl{D}_{k_{n+1}}\rangle | <\tau$.
		
		For every $\omega \in \Omega$, the map $g_{n+1}(\omega)\circ f\in \text{Lip}_0(M)$ is curve-flat, and satisfies that 
		$$ 
		(g_{n+1}(\omega)\circ f) (u)=0,~\text{ for all }u\in E_{n+1},
		$$
		because $g_{n+1}$ vanishes on $L_1 (\mathcal{A}_{k_n+1}, \mathcal{F} (f (E_{n+1})))$. In particular, keeping in mind the construction of $E_{n+1}$, we also observe that 
			\[
			\langle g_{n+1}, \widehat{f} \circ \cl{D}_{k_{i}} \rangle = \int_{\Omega}  \langle g_{n+1}(\omega) \circ f, \cl{D}_{k_{i}}(\omega)\rangle d\PP = 0
			\]
			for all $i=1,\ldots,n$.  
		Moreover:
		$$ \langle g_{n+1}, \widehat{f} \circ \cl{D}_{k_{n+1}} \rangle = \int_{\Omega}  \langle g_{n+1}(\omega) \circ f, \cl{D}_{k_{n+1}}(\omega)\rangle d\PP > \frac{1}{4}-\tau.$$
		Therefore, thanks to Lemma \ref{Lemma_Bate} and because $g_{n+1}$ is finitely valued, there is $h_{n+1}\in L_\infty( \A_{k_{n+1}+1} , \text{Lip}_0(M))$ with $\|h_{n+1}\|\leq 1$ and such that: 
		\begin{itemize}[leftmargin=*]
			\item $|\< h_{n+1}(\omega) , \delta(u) \> | < 2\tau$, for all $u\in M$ and all $\omega \in \Omega$,
			\item  $| \langle h_{n+1},\cl{D}_{k_i}\rangle | <\tau $ for all $i=1,\dots,n $, 
			\item $\langle h_{n+1}, \cl{D}_{k_{n+1}}\rangle >\frac{1}{4}-2\tau > \frac{1}{8}$, and
			\item  $h_{n+1}(\omega)$ is $\varepsilon_{n+1}$-flat at some radius $\theta_{n+1}$, which can be chosen so that $0<\theta_{n+1} \leq \theta_n$, for every $\omega \in \Omega$.
		\end{itemize}
		
		Now, define $\varphi_{n+1}=\varphi_n+h_{n+1}$. Let us check that $\varphi_{n+1}$ satisfies conditions (2) to (4). First, condition (2) is clear since $\varphi_n(\omega)$ is $(\sum_{i=1}^n\varepsilon_i)$-flat at radius $\theta_{n+1}$ and $h_{n+1}(\omega)$ is $\varepsilon_{n+1}$-flat at radius $\theta_{n+1}$. 
  
  Next, the estimation used in the proof of Proposition \ref{prop:CFimpliesDPCompactCase} shows that for every $\omega \in \Omega$, $\|\varphi_{n+1}(\omega)\|_L\leq\Big(1+\sum_{i=1}^n \varepsilon_i+4\frac{\tau}{\theta_n}\Big) $. Taking the maximum over $\omega \in \Omega$, we obtain (3) if $\tau$ is chosen small enough.
		
		Finally, let us examine the estimates in (4). If $i\leq n$:
		\begin{align*}
			\langle\varphi_{n+1},\cl{D}_{k_i}\rangle\geq\langle \varphi_{n},\cl{D}_{k_i}\rangle -|\langle h_{n+1},\cl{D}_{k_i}\rangle|>\frac{1}{4}-\sum_{i=1}^n\varepsilon_i-\tau.
		\end{align*}
		Simultaneously, we also have:
		\begin{align*}
			\langle \varphi_{n+1},\cl{D}_{k_{n+1}}\rangle \geq \langle h_{n+1},\cl{D}_{k_{n+1}}\rangle -|\langle \varphi_n,\cl{D}_{k_{n+1}} \rangle|>\frac{1}{4}-3\tau.
		\end{align*}
		Choosing $\tau$ so that $4\frac{\tau}{\theta_n}<\varepsilon_{n+1}$ yields conditions (3) and (4). The induction process is now complete.
		
		If we denote $h_1=\varphi_1$ then it is clear that $\varphi_i=\sum_{n=1}^i h_n$ for every $i \in \Natural$.
		By the point (3) above we have that $\norm{\varphi_i}\leq \frac98=:C$ for every $i$. Observe that 
		
		\[\begin{aligned}
				C\norm{\cl{D}_{k_i}}&\geq \mathbb E\left( \duality{\sum_{n=1}^i h_n,\cl{D}_{k_i}} \right)\\
				&= \mathbb E\left(\sum_{n=1}^i \duality{h_n,  \cl{M}_{k_i + 1}  } \right) - \mathbb E\left(\sum_{n=1}^i \duality{h_n,  \cl{M}_{k_i}  } \right) \\
				&= \mathbb E\left(\sum_{n=1}^i \duality{h_n, \mathbb{E}^{\mathcal A_{k_n+1}} ( \cl{M}_{k_i + 1} ) } \right) - \mathbb E\left(\sum_{n=1}^i \duality{h_n, \mathbb{E}^{\mathcal A_{k_n}} ( \cl{M}_{k_i} )  } \right) \\ 
				&= \mathbb E\left(\sum_{n=1}^i\duality{h_n,\cl{D}_{k_n}}\right) - \mathbb{E} \left( \sum_{n=1}^i \duality{  h_n, \cl{M}_{k_{n}+1} - \mathbb{E}^{\mathcal{A}_{k_{n}+1}} (\cl{M}_{k_i +1} )} \right) \\
				&\hspace{15em} + \mathbb{E} \left(\sum_{n=1}^i \duality{  h_n, \cl{M}_{k_{n}} - \mathbb{E}^{\mathcal{A}_{k_{n}}} (\cl{M}_{k_i } )} \right)
				\\ 
				&\geq  \frac{i}{8}   -  \left( \sum_{n=1}^i \norm{\cl{M}_{k_n +1} - \mathbb E^{\mathcal A_{k_n+1}}(\cl{M}_{k_{i}+1}) } + \sum_{n=1}^i \norm{ \cl{M}_{k_n} - \mathbb{E}^{\mathcal A_{k_n}}  (\cl {M}_{k_i} ) } \right) \\ 
				&\geq \frac{i}{8} -  \left( \sum_{n=1}^i \frac{1}{2^{k_n +1}} + \frac{1}{2^{k_n}} \right) > \frac{i}{8} -2. 
			\end{aligned}\]
		
		Since $i$ was arbitrary, this contradicts the assumption that $(\cl{M}_i)$ is bounded in $L_1(\Free(M))$.
	\end{proof}

	\subsection{Compact determination of RN operators}
	
	Just like Dunford-Pettis operators, we claim that the property ``$\hat{f}$ is Radon-Nikod\'ym" is compactly determined. The method of the proof follows the lines of the compact determination of the RNP for Lipschitz-free spaces (see \cite[Corollary 4.5]{AGPP}).

	\begin{proposition}\label{prop:RNisCompactlyDetermined}
		{Let $M, N$ be metric spaces, with $M$ complete,} and let $f \in \Lip_0(M,N)$. The property ``$\hat{f}$ is Radon-Nikod\'ym" is compactly determined, that is, $\widehat{f}:\F(M) \to \F(N)$ is Radon-Nikod\'ym if and only if for every compact subset $K$ of $M$ containing $0$, $\widehat{f\restricted_K}$ is Radon-Nikod\'ym.
	\end{proposition}
	
	\begin{proof}
		The ``only if" implication is clear.  We will prove the contrapositive of the ``if" direction.
		Suppose that $f : M \to N$ is a Lipschitz map such that $\widehat{f} : \F(M) \to \F(N)$ is not Radon-Nikod\'ym.
		By Lemma~\ref{Lemma:RN via unif integrable martingales}, 
		there exists a uniformly integrable martingale $(M_n)_n$ in $L_1(\lipfree{M})$ such that $(\widehat{f}(M_n))_n$ does not converge in $L_1(\lipfree{N})$. So there exists some $\delta>0$ such that
		
		\begin{equation}\label{eq:KP_0}
		    \limsup\limits_{n,m \to \infty}\|\widehat{f}(M_n)-\widehat{f}(M_m)\|_{L_1(\lipfree{N})} > \delta.
		\end{equation}
		
		By \cite[Proposition~4.2]{AGPP} (with $\eps = \frac{\delta}{4\|f\|_L}$, and using the fact that $(M_n)$ has the \emph{mean Kalton property} thanks to \cite[Proposition~4.3 and Remark~4.4]{AGPP}),
		there exist a compact subset $K$ of $M$ containing $0$, bounded linear operators $T_k:\lipfree{M}\to\lipfree{M}$ and a map $T:(M_n)_n\to L_1(\lipfree{K})$ such that
		\begin{equation}\label{eq:KP_1}
			\sup_{n\in\NN} \norm{M_n-T(M_n)}_{L_1(\lipfree{M})}\leq\frac{\delta}{4\|f\|_L}
		\end{equation}
		and 
		\begin{equation}\label{eq:KP_2}
			\lim_{k\to\infty}\sup_{n\in\NN} \norm{T_k(M_n)-T(M_n)}_{L_1(\lipfree{M})}=0 .   
		\end{equation}
		
		As demonstrated in the proof of \cite[Corollary~4.5]{AGPP}, we can observe that the sequence $(T(M_n))_n$ is a $L_1 (\mathcal{F}(K))$-bounded martingale using \eqref{eq:KP_1} {and \eqref{eq:KP_2}}. 
  
  Moreover, the sequence $(T_k(M_n))_n$ is uniformly integrable for each $k$ since each $T_k$ is bounded operator. Also, condition \eqref{eq:KP_2} then implies that $(T(M_n))_n$ is an $L_1(\lipfree{K})$-uniformly integrable martingale. But we have
		$$\limsup_{n,m \to \infty} \|\widehat{f}\circ T(M_n)-\widehat{f}\circ T(M_m)\|_{L_1(\lipfree{N})} > \frac{\delta}{2}$$
		by \eqref{eq:KP_0} and \eqref{eq:KP_1}, so $(\hat{f\restricted_K} ( T(M_n) ))_n$ cannot converge in $L_1(\lipfree{N})$. According to Lemma~\ref{Lemma:RN via unif integrable martingales} again, this shows that $\widehat{f\restricted_K} : \F(K) \to \F(N)$ is not Radon-Nikod\'ym, completing the proof.
	\end{proof}

Finally we prove the main theorem of the section.
\subsection{Proof of Theorem \ref{t:RN}}
\begin{proof}[Proof of Theorem \ref{t:RN}]
    By Theorem \ref{t:DP}, we only need to show $(1)\implies(2)$ and $(2)\implies (3)$. The first implication follows by Propositions \ref{prop:CFimpliesRNCompactCase} and \ref{prop:RNisCompactlyDetermined}. The contrapositive of the second implication is clear, since $L_1$ fails the RNP, and thus there exists an $L_1$-valued bounded martingale which does not converge almost surely. 
\end{proof}

	
	\section{Remarks on strong Radon-Nikod\'ym and representable Lipschitz operators}
	\label{section:RepresentableAndStrongRN}
   In this section, we discuss two classes of Lipschitz operators (strong Radon-Nikod\'ym and representable Lipschitz operators, see Section~\ref{ss:OperatorIdeals}) which, in some cases, can be added to the list of equivalences of Theorem \ref{t:flagship}.

 \subsection{Strong Radon-Nikod\'ym Lipschitz operators}\label{ss:StrongRNP}

    We start by showing that strong Radon-Nikod\'ym and Radon-Nikod\'ym Lipschitz operators are indeed the same provided the underlying metric space in the domain is compact and countably $1$-rectifiable.
    
	Following \cite[Definition~3.2.14, Page~251]{Federer} (see also \cite[Section~3.2]{FreemanGartland1}), a metric space $M$ is \emph{countably 1-rectifiable} if there exist a countable collection of subsets $A_i \subset \R$ and Lipschitz maps $\gamma_i : A_i \to M$ such that $\mathcal H^1\big(M \setminus \big( \bigcup_{i} \gamma_i(A_i) \big) \big) = 0$. Clearly, any subset of $\R$ is countably 1-rectifiable. Notice also that $M$ is purely 1-unrectifiable if and only if every countably 1-rectifiable subset of $M$ is $\mathcal H^1$-negligible.  
	
    \begin{prop}\label{p:RN_eq_SRN_in_1_rectifiable_spaces}
    Let $M$ be a compact metric space which is countably 1-rectifiable, let $N$ be any metric space and let $f \in \Lip_0(M,N)$. The following are equivalent.
    \begin{enumerate}
        \item $f(M)$ is $\H^1$-negligible.
        \item $\hat{f}$ is strong Radon-Nikod\'ym.
        \item $\hat{f}$ does not fix a copy of $L_1$.
    \end{enumerate}
    \end{prop}
    \begin{proof}
        $(1) \implies (2)$. Since $M$ is compact and $\H^1(f(M)) = 0$, $f(M)$ is complete and purely 1-unrectifiable. Thus, \cite[Theorem~C]{AGPP} implies that $\F(f(M))$ has the Radon-Nikod\'ym property. Since $\widehat{f}(B_{\F(M)})$ is contained in $\F(f(M))$, this yields that $\widehat{f}$ is strong Radon-Nikod\'ym.

        $(2)\implies (3)$. This implication is true in general, since every strong Radon-Nikod\'ym operator is Radon-Nikod\'ym, and, as we discussed in the proof of Theorem~\ref{t:RN}, such an operator cannot fix a copy of $L_1$.
        
        $(3)\implies (1)$. We do not use the compactness assumption on $M$ for this implication.  Suppose that $\H^1(f(M)) > 0$. Since $M$ is countably 1-rectifiable, there exist a countable collection $(A_i)_{i \in I}$ of subsets of $\R$ and Lipschitz maps $\gamma_i : A_i \to M$ such that $\mathcal H^1\big(M \setminus \big( \bigcup_{i\in I}\gamma_i(A_i) \big) \big) = 0$. Since $f$ is Lipschitz, 
		\begin{align*}
			\mathcal H^1\Big(f(M) \setminus f\big( \bigcup_{i\in I}\gamma_i(A_i)\big) \Big)  &\leq \mathcal H^1\Big(f\big(M \setminus \big( \bigcup_{i\in I}\gamma_i(A_i) \big) \big) \Big) \\
			&\leq \|f\|_L \;  \mathcal H^1\big(M \setminus \big( \bigcup_{i\in I}\gamma_i(A_i) \big) \big) = 0.
		\end{align*}

  Therefore, by $\sigma$-additivity of $\mathcal H^1$, there exists $i \in I$ such that $\H^1(f \circ \gamma_i(A_i)) > 0$. By Kirchheim's lemma, there exists a compact subset $B$ of  $A_i$ such that $\lambda(B)>0$ and $(f\circ \gamma_i)\restricted_B$ is bi-Lipschitz. Thus $\gamma_i\restricted_B : B \to \gamma_i(B)$ and $f\restricted_{\gamma_i(B)} : \gamma_i(B) \to f\circ \gamma_i(B) $ are both bi-Lipschitz as well, and in particular $\widehat{\gamma_i\restricted_B}$ and $\widehat{f\restricted_{\gamma_i(B)}}$ are isomorphisms onto their respective range.  On the other hand, since $\lambda(B)>0$, Godard's work \cite{Godard} implies that $\F(B)$ contains an isomorphic copy $E$ of $L_1$. Thus $\widehat{\gamma_i\restricted_B}(E)$ is also an isomorphic copy of $L_1$. Restricting $\hat{f}$ to this copy of $L_1$ yields the desired result.
    \end{proof}

\begin{remark}\label{Remark:H1_negligible_DP}
    It is worth noting that the equivalence between $(1)$ and $(3)$ in Proposition \ref{p:RN_eq_SRN_in_1_rectifiable_spaces} holds 
    when $M$ is only assumed to be complete and not necessarily compact. 
    Indeed, we already pointed out that our proof of $(3)\implies (1)$ does not need the compactness of $M$. 
    Conversely, let $M$ be a complete countably $1$-rectifiable set, and suppose that $f(M)$ is $\mathcal{H}^1$-negligible for a given 
    $f\in \Lip_0(M,N)$. 
	It is clear that $f(K)$ is also $\mathcal{H}^1$-negligible for any compact $K \subset M$. 
    Then, 
    (if moreover  $0\in K$)
    by Proposition \ref{p:RN_eq_SRN_in_1_rectifiable_spaces} and Theorem \ref{t:flagship}, we see that $\widehat{f \restricted_K}$ is Dunford-Pettis. 
    As $K$ can be any compact subset of $M$ 
    such that $0\in K$, 
    we can apply Proposition \ref{prop:DPisCompactlyDetermined} to conclude that $\hat{f}$ is Dunford-Pettis, and thus cannot fix a copy of $L_1$. 

 Notice however that completeness of $M$ is necessary. 
 For instance, consider $M=\mathbb Q \cap [0,1]$, $N=[0,1]$ and $f=\Id_M$. Then $\widehat{f}$ is conjugate to the identity on $L_1$; so it fixes a copy of $L_1$ while the set $f(M)=M$ is countable and thus $\H^1$-negligible. 
\end{remark}

Thanks to Theorem \ref{t:flagship}, Proposition \ref{p:RN_eq_SRN_in_1_rectifiable_spaces} shows in particular that, if $M$ is a compact countably $1$-rectifiable metric space, then $\hat{f}\colon \F(M)\rightarrow \F(N)$ is Radon-Nikod\'ym if and only if is strong Radon-Nikod\'ym. However, this is not true if $M$ is only assumed to be complete, as shown by the following example.

\begin{example}\label{Example:RN_not_SRN}
    Let $N$ be a complete, separable and non-purely $1$-unrectifiable metric space, with diameter not greater than $1$ (e.g. the unit interval in $\R$). Let $(x_k)_k$ be a countable dense subset of $N$, and let $\varphi\colon \N\rightarrow \tilde{\N}=\N^2\setminus \{(n,n)\colon n\in\N\}$ be an enumeration. Write $\varphi(n) = ( \varphi(n)_1, \varphi(n)_2)$ for each $n \in \mathbb{N}$. Consider $T$ the unique $\R$-tree with only one branching point $0\in T$ and with branches $B_n:=[0,1+d(x_{\varphi(n)_1},x_{\varphi(n)_2})]$ for every $n\in \N$. Denote by $e_{(n,1)}$ the point in $B_n$ at distance $1$ from $0$, and by $e_{(n,2)}$ the point in $B_n$ at distance $1+d(x_{\varphi(n)_1},x_{\varphi(n)_2})$ from $0$ and thus at distance $d(x_{\varphi(n)_1},x_{\varphi(n)_2})$ from $e_{(n,1)}$ . Let $M=\bigcup_{n\in \N}\{e_{(n,1)},e_{(n,2)}\}\cup \{0\}\subset T$, with the $\R$-tree distance inherited from $T$. It is clear that $M$ is a complete metric space. Then, for every pair {$k_1\neq k_2\in {\mathbb N}$}, if we choose $n=\varphi^{-1}(k_1,k_2)$, it holds that $d_N(x_{k_1},x_{k_2})=d_M(e_{(n,1)},e_{(n,2)})$. Then we define $f\colon M\rightarrow N$ as $f(0)=0$, $f(e_{(n,1)})=x_{\varphi(n)_1}$, and $f(e_{(n,2)})=x_{\varphi(n)_2}$. Since the diameter of $N$ is not greater than $1$, it follows from the tree structure of $M$ that $f$ is $1$-Lipschitz. Then, using that $(x_k)_k$ is dense in $N$, it easily follows that $C_f=(\hat{f})^*$ is a linear isometry. Hence, $\hat{f}$ is a linear quotient from $\F(M)$ onto $\F(N)$. This implies that $\hat{f}$ is not strong Radon-Nikod\'ym, since $\F(N)$ fails the RNP. However, $\hat{f}$ is Radon-Nikod\'ym because $M$ is purely $1$-unrectifiable (so, $\mathcal{F}(M)$ has the RNP), and thus {$f$} is trivially curve-flat.
    Finally, since $M$ is countable, we have $\H^1(M)=0$ and so $M$ is countably 1-rectifiable.

\end{example}

	\subsection{Representable Lipschitz operators}
{In this subsection we consider the special case of complete metric spaces $M$ such that $\F(M)$ is isomorphic to $L_1(\mu_M)$, for some finite measure $\mu_M$. So assume that $\Phi: L_1(\mu_M) \to \F(M)$ is an isomorphism. Then, for any metric space $N$ and any $0$-preserving Lipschitz map $f:M \to N$, the Lipschitz operator $\widehat{f} : \F(M) \to \F(N)$ is conjugate to a bounded operator $T_f=\widehat{f} \circ \Phi: L_1(\mu_M) \to \F(N)$. We can thus study when the operator $T_f$ is representable, based on the properties of $f$. Our main result in this subsection is based on the following general fact. Although we did not find it explicitly stated in the literature, this is well-known among specialists. Let us just say that it is an operator version of \cite[Theorem~III.1.5]{Diestel}. It can easily be obtained by combining the characterization of RN operators due to Edgar \cite[p. 461]{Edgar} and \cite[Lemma~III.1.4]{Diestel} for one direction and by using the Lewis-Stegall factorization theorem of representable operators (see e.g. the proof of \cite[Theorem~III.1.8]{Diestel}) for the converse. 
 
 \begin{lemma} \label{Lemma:RepresentableIffRNP}
		Let $(\Omega , \Sigma , \mu)$ be a finite measure space and $Y$ be a Banach space. Then a bounded operator $T : L_1(\mu) \to Y$ is representable if and only if it is Radon-Nikod\'ym.
	\end{lemma}
We now deduce the following corollary. Note that several more items could be added to the equivalence therein according to Theorem \ref{t:flagship}.

\begin{corollary} \label{cor:DPPiffRepiffRNP}
Let $M$ be a complete pointed  metric space and assume that there exists an isomorphism $\Phi$  from  $L_1(\mu_M)$ onto $\F(M)$, where $\mu_M$ is a finite measure. Let $N$ be a pointed metric space and let {$f \in \Lip_0 (M,N)$} be given. Then the following assertions are equivalent.
		\begin{enumerate}[$(i)$]
			\item $T_f = \widehat{f}\circ \Phi : L_1(\mu_M) \to \F(N)$ is representable.
			\item $\widehat{f}$ is Radon-Nikod\'ym.
		\end{enumerate}
	\end{corollary}
	\begin{proof}
Applying Lemma \ref{Lemma:RepresentableIffRNP} we obtain that $T_f$ is representable if and only if $T_f$ is Radon-Nikod\'ym. Since $\hat{f}=T_f\circ \Phi^{-1}$ and Radon-Nikod\'ym operators form an operator ideal, we obtain the desired result.
\end{proof}

\begin{remark}
The assumptions of the above corollary apply to separable $\R$-trees \cite{Godard} and even to quasiconformal trees \cite{FreemanGartland1}. We also refer to \cite{FreemanGartland2} and \cite{Gartland} for more examples of metric spaces whose free spaces are isomorphic to $L_1$. 
\end{remark}}

\subsection{Remarks on bounded operators from \texorpdfstring{$\Free(M)$}{F(M)} to \texorpdfstring{$X$}{X}}

Here we are going to show that  there is no analog of Theorem~\ref{t:flagship} 
 for the linearization $\overline{f}:\mathcal{F}(M) \rightarrow X$ of a Lipschitz map $f:M \to X$, where $M$ is a metric space and $X$ is a Banach space.
 This should be no surprise as every bounded linear operator $T:\Free(M) \to X$ is of the form $\overline{f}$ for some $f \in \Lip_0(M,X)$.

\begin{proposition}\label{p:NoAnalogue}
    Let $M$ be a complete metric space, $X$ be a Banach space and $f\in \Lip_0(M,X)$. Then
\[
\xymatrix{
	{f\in \CF} \ar@{=>}[rr]  \ar@{=>}[dd] \ar@{=>}[rrdd]  & & {\overline{f}\in \DP} \ar@{=>}[dd]  \\
	\\
	{\overline{f}\in \RN} \ar@{=>}[rr] & & {\overline{f} \mbox{ does not fix copy of } L_1},
}
\]
and all the implications not shown on the diagram are false in general.
\end{proposition}

For the proof we will need the following proposition of independent interest.

\begin{prop}\label{prop:Space(operator_ideal)} Let $\mathcal I$ be an operator ideal and let $Space(\mathcal I)$ be the class of Banach spaces $X$ such that $\Id_X \in \mathcal I$. 
    Then $X \in Space(\mathcal I)$ implies $\beta_X \in \mathcal I$.
    If, in addition, $Space(\mathcal I)$ is separably determined, then $\beta_X \in \mathcal I$ implies $X \in Space(\mathcal{I})$. In particular, $X$ has the SP if and only if $\beta_X$ is DP, and $X$ has the RNP if and only if $\beta_X$ is RN.
\end{prop}

\begin{proof}
        The first implication follows since $\beta_X=\Id_X\circ \beta_X$ and $\mathcal I$ is an operator ideal.

        To see the second implication, it is enough to show that every separable subspace $Y$ of $X$ belongs to $Space(\mathcal I)$. 
       Fix a separable subspace $Y$ of $X$ and observe the following commutative diagram 
        $$\xymatrix{
	\F(Y) \ar[r]^{\widehat{\iota}} \ar[d]_{\beta_Y}  & \F(X) \ar[d]^{\beta_X} \\
	Y \ar[r]_{{\iota}} & X,
}
$$
where $\iota : Y \hookrightarrow X$ is the canonical linear embedding, and $\beta_X$ and $\beta_Y$ are linear quotient maps which are the left inverses of $\delta_X$ and $\delta_Y$, respectively. 
That is, $\iota \circ \beta_Y = \beta_X \circ \widehat{\iota}$. Write $E = \iota(Y)$ and notice from \cite[Theorem 3.1]{GK03} that there exists a Godefroy-Kalton isometric linear lifting $T: Y \rightarrow \mathcal{F}(Y)$ of the map $\beta_Y$, i.e. 
           \[
    \Id_Y = \beta_Y \circ T.
    \]
Noting $\iota^{-1} : E \rightarrow Y$ is bounded as well, we conclude 
\[
\Id_Y = \iota^{-1} \circ \iota \circ \beta_Y \circ T = \iota^{-1} \circ \beta_X \circ \widehat{\iota} \circ T.
\]
    Since $\beta_X\in \mathcal I$, we have that $\Id_Y\in \mathcal I$, equivalently, $Y\in Space(\mathcal I)$.

 The last part of the proposition holds as SP$=Space(DP)$, RNP$=Space(RN)$ and both classes SP and RNP are separably determined.
    \end{proof}

\begin{proof}[Proof of Proposition~\ref{p:NoAnalogue}]
    Taking into account that $\overline{f}=\beta_X \circ \widehat{f}$, it follows from Theorem \ref{t:flagship} that if $f$ is curve-flat, then $\overline{f}$ is DP, RN and does not fix any copy of $L_1$. 
    The converses of the three implications are easily seen to be false. 
    Indeed, if $f:\R \to \R$ is the identity, then it is not curve-flat, while $\overline{f}:\mathcal F(\R) \to \R$ is of rank one and therefore DP, RN and not fixing copy of $L_1$. 

    It is a general fact, that neither DP nor RN operators fix copies of $L_1$.

    Next, consider $\Id_{\ell_2}: \ell_2 \to \ell_2$ the identity map. 
    Since $\ell_2$ has the RNP, the linearization $\overline{\Id_{\ell_2}}:\mathcal F(\ell_2) \to \ell_2$ is RN (and fixes no copy of $L_1$). 
    However, $\overline{\Id_{\ell_2}}$ is nothing else than the quotient map $\beta_{\ell_2}$; so it is not DP by Proposition~\ref{prop:Space(operator_ideal)}.

    Similarly, let $X$ be the Bourgain-Rosenthal space which has the SP but lacks the RNP~\cite{BR1980Israel}.
    Then $\beta_X$ is DP (and therefore fixes no copy of $L_1$) but is not RN by Proposition~\ref{prop:Space(operator_ideal)}.
\end{proof}

 {\begin{remark} 
 The last part of the above proof can be achieved by a much more elementary example.
  Let $(A_n)_{n=1}^\infty$ be an enumeration of the dyadic intervals in $[0,1]$ and define $T:L_1=L_1([0,1]) \to c_0$ by $T\varphi=(\int_{A_n}\varphi\,d\lambda)_{n=1}^\infty$. Since weakly null sequences in $L_1$ are uniformly integrable, we easily see that $T$ is DP. On the other hand, $T$ is not RN. Indeed, consider the $L_1$-valued measure $G(E)=\xi_E$, for $E$ Borel subset of $[0,1]$, which is absolutely continuous with respect to $\lambda$. Assume that $T$ is representable, so that there exists $g=(g_n)_{n=1}^\infty \in L_1([0,1];c_0)$ such that for all $E$ Borel, $T \circ G(E)=\int_Eg\,d\lambda$. Then, for each $n\in \N$, $g_n=\xi_{A_n}$ and $\limsup_n g_n(t)=1$ for any $t\in [0,1]$, which is a contradiction. Thus $T$ is not representable. Finally, identifying $L_1([0,1])$ with $\mathcal F([0,1])$, we can view $T$ as $\overline{f}$, where $f:[0,1] \to c_0$ is the Lipschitz map defined by $f(t)=(\lambda([0,t]\cap A_n))_{n=1}^\infty$. Thus we have that $\overline{f}$ is DP but not RN. 
\end{remark}
}
	
	\section{Quotient metric spaces and factorizations}
	\label{section:factorization}
	
	In general, given an operator $T\colon X\rightarrow Y$ between Banach spaces, it is clear that if $T$ factors through a Schur (respectively RNP) space, then $T$ is Dunford-Pettis (respectively Radon-Nikod\'ym). 
    However, the converse is not necessarily true.

In our context, it is still a natural question whether the conditions $\widehat{f}\in Op(SP)$ (resp. $\widehat{f}\in Op(RNP)$), i.e. ``$\widehat{f}$ factors through a space with the SP (resp. the RNP)'', can be added to the list of equivalences in Theorem \ref{t:flagship}. Taking into account the characterization of Lipschitz-free spaces with the SP (resp. the RNP), we can also ask whether the stronger condition ``$f$ factors through a complete purely 1-unrectifiable metric space" can also be included in Theorem \ref{t:flagship}. In Theorem~\ref{t:MainOfSection6}, we will characterize the compact metric spaces $M$ for which this stronger factorization condition holds, for every range space $N$. Then, we will present a concrete example showing that the answer is negative in general (Example~\ref{cor:Example_CF_not_Factp-1-u}). 
Although the ``factorization through Schur spaces'' question remains open, we will provide a few related observations in the last two subsections. 


\subsection{Factoring through p-1-u spaces: the compact case}

\sloppy In what follows, $Fact_{p1u}(M,N)$ stands for the set of those $f \in \Lip_0(M,N)$ which factor through a complete p-1-u space. 
When $M$ is compact, we can omit to stress the completeness of the p-1-u factor as that is automatic.
We always have $Fact_{p1u}(M,N) \subset \CF_0(M,N)$ as $\Id_P:P\to P$ is curve-flat for any p-1-u metric space $P$.
For maps $f \in \Lip_0(M,N)$ and $g \in \Lip_0(M,P)$ we will say that \emph{$f$ factors through $g$} if there exists $h \in \Lip_0(P,N)$ such that $f=h\circ g$.
We will denote $Fact_g(M,N)$ the set of such maps.
Given a metric space $M$, a particular role is going to be played by the metric space $M_{\lip}$ defined as the canonical quotient of $M$ where we identify the points $x,y \in M$ such that $d_{\lip}(x,y)=0$ 
where 
	$$d_{\lip}(x,y) := \sup\{|f(x)-f(y)| \, : \, f \in B_{\lip_0(M)}\}.$$
We equip $M_{\lip}$ with the metric $d_{\lip}([x],[y])=d_{\lip}(x,y)$.
Let us denote $\pi:M \to M_{\lip}$ the canonical quotient map defined by $\pi(x)=[x]$ for every $x \in M$. Clearly $\|\pi\|_L \le 1$.
It is well known and easy to see that, given any metric space $N$, every $f\in \lip_0(M,N)$ factors through $\pi:M \to M_{\lip}$. 
In other words, we always have $\lip_0(M,N) \subset Fact_{\pi:M\to M_{\lip}}(M,N)$. Since $M$ is compact, \cite[Corollary~8.13]{Weaver2} combined with \cite[Theorem A]{AGPP} imply that $M_{\lip}$ is p-1-u. Thus, we arrive at the following observation.

\begin{prop}
    \label{p:cpt-ulf}
Let $M$ be a compact metric space and let $N$ be any metric space. If $f \in \lip_0(M,N)$, then $f$ factors through $M_{\lip}$ which is compact and p-1-u. 
\end{prop}

On the other hand, it follows from the results in~\cite{AGPP} that, for compact $M$, given any p-1-u metric space $P$, every $f \in \Lip_0(M,P)$ factors through $\pi:M \to M_{\lip}$.

\begin{lem}\label{l:p-1-uTarget}
    Let $M$ be compact and $P$ be p-1-u. 
    Let $f\in \Lip_0(M,P)$.
    Then $f$ factors through $\pi:M\to M_{\lip}$.
    More precisely, there exists $f_{\lip} \in \Lip_0(M_{\lip},P)$ such that $\norm{f_{\lip}}_L\leq \norm{f}_L$ and $f=f_{\lip}\circ \pi$.
\end{lem}
\begin{proof}
    This follows from \cite[Proposition 5.8]{AGPP} and~\cite[Theorem 5.9]{AGPP}.
    Here we provide a more direct proof:
    We define $f_{\lip}:M_{\lip} \to P$ by $f_{\lip}([x])=f(x)$.
    Let $x,y \in M$.
    We have 
    \[
    \begin{aligned}
        d(f_{\lip}([x]),f_{\lip}([y]))&=d(f(x),f(y))\\
        &= \sup\set{\varphi(f(x))-\varphi(f(y)):\varphi \in B_{\lip_0(P)}}\\
        &\leq \norm{f}_L d_{\lip}([x],[y])
    \end{aligned}
    \]
    where we have used~\cite[Theorem A]{AGPP} and the fact that $\varphi \circ f \in \lip_0(M)$. 
    This inequality shows at once that $f_{\lip}$ is well defined and that $\norm{f_{\lip}}_L\leq \norm{f}_L$.
\end{proof}

\begin{prop}\label{p:p-1-u-factorization-characterization}
    Let $M$ be a compact metric space and let $N$ be any metric space. Then 
     \[\displaystyle Fact_{p1u}(M,N)=Fact_{\pi:M\to M_{\lip}}(M,N).\]
    
 \sloppy   Moreover, there exists a canonical bijective correspondence between this set and $\Lip_0(M_{\lip},N)$. 
   Namely, the pre-composition operator $C_\pi: \Lip_0(M_{\lip},N) \to Fact_{p1u}(M,N)$, given by $u \mapsto u \circ \pi$, establishes this correspondence.
   It moreover satisfies $\norm{C_\pi(u)}_L\leq \norm{u}_L$ for every $u \in \Lip_0(M_{\lip},N)$ and 
   \[ \forall f \in Fact_{p1u}(M,N), \quad  \norm{C_\pi^{-1}(f)}_L=\inf\norm{h}_L\norm{g}_L,\] 
    where the infimum is taken over all factorizations $f=h\circ g$ through a complete p-1-u metric space.

\end{prop}

\begin{proof}

Assume first that $f\in \Lip_0 (M,N)$ factors through $\pi:M \to M_{\lip}$. As mentioned before Proposition \ref{p:cpt-ulf}, the metric space $M_{\lip}$ is p-1-u; so $f \in Fact_{p1u}(M,N)$. 

Assume now $f \in Fact_{p1u}(M,N)$ and that $P$ is a complete p-1-u metric space such that $f=h\circ g$, where $g\in \Lip_0(M,P)$ and $h\in \Lip_0(P,N)$. Since $M$ is compact, we can write $f=h\circ g_{\lip}\circ \pi$ where $\norm{g_{\lip}}_L\leq \norm{g}_L$ by Lemma~\ref{l:p-1-uTarget}.
Thus $f \in Fact_{\pi:M\to M_{\lip}}(M,N)$.

\sloppy Now observe that, since $\pi:M \to M_{\lip}$ is surjective, the map $C_\pi:u \in \Lip_0(M_{\lip},N) \mapsto u \circ \pi \in Fact_{p1u}(M,N)$ is injective and $\norm{C_\pi(u)}_L\leq \norm{u}_L$, since $\|\pi\|_L \leq 1$. The argument above shows that $C_\pi$ is also surjective with
$C_\pi^{-1}(f)=h\circ g_{\lip}$ with $\norm{C_\pi^{-1}(f)}_L\leq \norm{h}_L\norm{g}_L$.
The opposite inequality is clear as $f=C_\pi^{-1}(f)\circ \pi$ and $\|\pi\|_L \leq 1$.
\end{proof}

As we mentioned in Section~\ref{ss:CF}, the space $\CF_0 (M,Y)$ is a Banach space. Thus, the following is obtained as a consequence of Proposition~\ref{p:p-1-u-factorization-characterization} and the open mapping theorem.

\begin{cor}\label{cor:CompositionIsomorphism}
    Let $M$ be a compact metric space and $Y$ be a Banach space. Then $C_\pi: \Lip_0(M_{\lip},Y) \to Fact_{p1u}(M,Y)$ is linear, bounded and pointwise-to-pointwise continuous. If moreover  $\CF_0(M,Y) \subset Fact_{p1u}(M,Y)$, then $C_\pi$ is an isomorphism from $\Lip_0(M_{\lip},Y)$ into $\CF_0(M,Y)$.
\end{cor}
\begin{proof}
    The first part of the corollary follows immediately from Proposition~\ref{p:p-1-u-factorization-characterization}. If $\CF_0(M,Y) \subset Fact_{p1u}(M,Y)$ then $\CF_0(M,Y)=Fact_{p1u}(M,Y)$, since the reverse inclusion is always true.
    We have that $\norm{C_\pi}\leq \norm{\pi}_L$ and
    Proposition~\ref{p:p-1-u-factorization-characterization} claims that this map is a bijection from $\Lip_0(M_{\lip},Y)$ to $Fact_{p1u}(M,Y)$.
    Thus, by the open mapping theorem, $C_\pi$ is an isomorphism.
\end{proof}

Let us turn to conditions on $M$ guaranteeing $\CF_0(M,N)\subset Fact_{p1u}(M,N)$. The next theorem shows that when looking for curve-flat functions which do not factor through a complete p-1-u space, it is enough to consider functions with real values.
Also, this theorem relates the inclusion $\CF_0(M,N) \subset Fact_{p1u}(M,N)$ to the condition that the following natural quotient metric space is p-1-u. Let $(M,d)$ be a compact metric space and define 
\[
	\forall x, y \in M, \quad d_{\CF}(x,y) :=  \sup\{|f(x)-f(y)| \, : \, f \in B_{\CF_0(M)}\}.
\]
We denote  $M_{\CF}$ the canonical quotient space of $M$ where we identify the points $x, y \in M$ such that $d_{\CF} (x,y) =0$, which we equip with the metric $d_{\CF} ([x],[y]) = d_{\CF} (x,y)$.
While, on the first sight, the metric space $M_{\CF}$ seems to be defined using behavior of Lipschitz functions on $M$, in fact it admits a fully metric definition. This has been proved in~\cite[Proposition 5.22]{AGPP}. In particular, the metric spaces $M$ such that $M_{\CF}$ is p-1-u are those where the iterative process of ``removing curve fragments'' stabilizes after at most 1 step.

\begin{thm}\label{t:MainOfSection6} 
    Let $M$ be a compact metric space. Then {the following assertions are equivalent.}
    \begin{enumerate}
        \item $\CF_0(M) \subset Fact_{p1u}(M,\Real)$,
        \item $\CF_0(M,N) \subset Fact_{p1u}(M,N)$ for every metric space $N$,
        \item The metric space $M_{\CF}$ is compact p-1-u. 
        \item The metric spaces $M_{\CF}$ and $M_{\lip}$ are Lipschitz equivalent.
        \item $B_{\CF_0(M)} \subset \alpha\overline{B_{\lip_0(M)}}^{w^*}$ for some $\alpha>0$.
        \item $\CF_0(M)\subset\overline{\lip_0(M)}^{w^*}$.
    \end{enumerate}
    Moreover, under any of these conditions, the constant $\alpha$ in (5) can be taken equal 1 and the Lipschitz equivalence in (4) is an isometry. Also the inclusions in items (1), (2), (5) with $\alpha=1$ and (6) are equalities.
\end{thm}

\begin{proof}
(5) $\implies$ (4). Recalling the definition of $d_{\lip}$ and $d_{CF}$, observe that, under assumption (5), we have that for $x \neq y \in M$
\[
d_{\lip} (x,y) \leq d_{\CF} (x,y) \leq \alpha\, d_{\lip} (x,y). 
\]
This implies that $M_{\CF}$ and $M_{\lip}$ are bi-Lipschitz equivalent.  

(4) $\implies$ (3). 
Indeed, by Proposition \ref{p:cpt-ulf}, $M_{\CF}$ is compact and p-1-u. 

(3) $\implies$ (2) is clear, as for any metric space $N$, any $f \in \CF_0 (M,N)$ factors through $M_{\CF}$ in a canonical way. 

As (2) $\implies$ (1) is trivial, we claim that (1) $\implies$ (5). Let $f \in B_{\CF_0} (M)$ be given. By Corollary~\ref{cor:CompositionIsomorphism} there is $u \in \Lip_0(M_{\lip})$ such that $\norm{u}_L\leq \norm{C_\pi^{-1}}$ and $C_\pi(u)=f$.
    Let $\alpha:=\norm{C_\pi^{-1}}$.
    Since $M_{\lip}$ is compact p-1-u, we have $\lip_0(M_{\lip})^{**}=\Lip_0(M_{\lip})$.
    So there is a sequence $(u_n)\subset \alpha B_{\lip_0(M_{\lip})}$ which converges pointwise to $u$.
    Thus, by Proposition~\ref{p:p-1-u-factorization-characterization}, the sequence $u_n\circ \pi$ converges pointwise (and, being bounded, also weakly$^*$) to $u\circ \pi=f$ which concludes the proof as $u_n\circ \pi \in \lip_0(M)$.

Further (5) clearly implies (6). In order to prove (6) $\Rightarrow$ (5), observe that Petun{\={\i}}n-Pl{\={\i}}{\v{c}}ko theorem \cite{Petunin74} (see also \cite{God87})  yields the following fact: Let $X$ be a separable Banach space and $Z\subseteq X^*$ be a separable normed-closed subspace contained in the set of norm-attaining functionals. 
Then $Z^{**}=\overline{Z}^{w^*}$ isometrically.
Indeed, denote $Y=X/Z_\perp$.
Then $Y^*=\overline{Z}^{w^*}$ isometrically.
Moreover $Z$ clearly separates points of $Y$. Also $Z$ is a subset of norm-attaining functionals on $Y$. This last claim follows from the facts that $\norm{z}_{Y^*}=\norm{z}_{X^*}$ and $Z \subset NA(X)$. It then follows from Petun{\={\i}}n-Pl{\={\i}}{\v{c}}ko theorem that $Y$ is isometric to $Z^*$, and therefore $Z^{**}$ is isometric to $\overline{Z}^{w^*}$. Thus applying the above fact to $Z=\lip_0(M)$ -- notice that $\lip_0(M)$ is contained in the set of norm attaining functionals on $\F(M)$ thanks to \cite[Lemma~2.3]{DaletCompact} -- we get $\CF_0(M)\subseteq \lip_0(M)^{**}$ and (5) follows with $\alpha=1$ by Goldstine's theorem.

The very last claim follows from the facts $Fact_{p1u} (M, N) \subseteq \CF_0 (M,N)$ in general and the weak$^*$-closedness of $\CF_0 (M)$ due to Proposition~\ref{p:CFwsClosed}.
\end{proof}

We continue with some more concrete conditions on $M$ that are sufficient for the inclusion $\CF_0(M,N) \subset Fact_{p1u}(M,N)$ to hold for every metric space $N$.
\begin{prop} \label{p:cpt-c1r}
		Let $M$ be compact and countably 1-rectifiable, and $N$ be any metric space. If $f\in \CF_0(M,N)$ then $f(M)$ is p-1-u. In particular, $f$ factors through a compact p-1-u space.  
  	\end{prop}
 \begin{proof}
Theorem~\ref{t:flagship} implies that $\hat{f}$ does not fix any copy of $L_1$ and so $\mathcal{H}^1 (f(M))=0$ by Proposition~\ref{p:RN_eq_SRN_in_1_rectifiable_spaces}.     
Thus $f(M)$ is compact and p-1-u as claimed.
 \end{proof}
 \begin{remark}\label{r:cpt-c1r}
 When $M$ is compact and countably 1-rectifiable we have that
  $B_{\lip_0(M)}$ is weak$^*$-dense in $B_{\CF_0(M)}$. 
  This can be of course seen from Theorem~\ref{t:MainOfSection6} combined with Proposition~\ref{p:cpt-c1r}.
  But we can also get a bit stronger approximation:
  since $\overline{f(A)}=f(A)$ by compactness, \cite[Lemma 4.8]{GPP23} guarantees that there exists a sequence $(\varphi_n)_n$ of locally constant functions in $B_{\Lip_0 (A)}$ such that $\varphi_n \rightarrow f$ in the weak$^*$ topology. 
 \end{remark}

\subsection{A curve-flat function which does not factor through a p-1-u space}

The following provides a negative answer to the question of whether the condition ``$f \in Fact_{p1u}$" can be added to the list of equivalences in Theorem \ref{t:flagship}.

\begin{example}
\label{cor:Example_CF_not_Factp-1-u}
\textit{There exists a compact $M$ such that $\CF_0(M) \not\subset Fact_{p1u}(M,\Real)$.
In fact we have that each non-constant $f\in \CF_0(M)$ does not factor through any p-1-u space.}
\end{example}
\begin{proof}
    The example of such $M$ comes from \cite[Example~5.25]{AGPP}.
	Consider the standard middle thirds Cantor set $\mathcal C \subset [0,1]$ and let $\beta := \log_3(2)$ be the Hausdorff dimension of $\mathcal C$. 
 The Cantor function $f: \mathcal C^\beta \to [0,1]$ is monotone, surjective, and 1-Lipschitz (\cite[Proposition 10.1]{Cantor}), where $\mathcal C^\beta$ denotes the snowflake space. 
 Let $(M,d)$ be a metric space obtained by ``filling in the gaps" of $\mathcal C^\beta$ with geodesics. In fact there are many ways to do the above but we equip $M$ by the metric which is the largest possible among all such metrics. We will not use this particularity until Remark~\ref{r:ExampleNotCounter}.
The Cantor function $f$ extends to a 1-Lipschitz map $f: M \to [0,1]$ that is constant on each geodesic. 
Observe that $f$ is curve-flat because $\mathcal C^\beta$ is purely 1-unrectifiable and $f$ is constant on each of the countably many geodesics. 
On the other hand, note that every locally flat function is constant. Hence, $M_{\lip}=\{0\}$ and so $f\notin Fact_{\pi:M \to M_{\lip}}(M,\Real)$. 
\end{proof}

\begin{remark}\label{remark:M_CF_isometric_to_interval}
  It might be an interesting research project to study for which (non-p-1-u) compacts $K$ there exists compact $M$ such that  $M_{\CF}=K$.
			For instance, it can be shown that $M_{\CF}$  corresponding to the space from Example~\ref{cor:Example_CF_not_Factp-1-u} is isometric to the unit interval $[0,1]$ with the standard metric. 
   Indeed, if $f\colon M\rightarrow[0,1]$ is the extended Cantor function described in Example~\ref{cor:Example_CF_not_Factp-1-u}, the function $f_{\CF_0}\colon M_{\CF}\rightarrow [0,1]$, where $f_{\CF_0}([x])=f(x)$ for all $[x]\in M_{\CF}$, is the desired isometry. 
   To see this, note first that, since $f$ is curve-flat and $1$-Lipschitz, we have that $d([x_1],[x_2])=d_{\CF}(x_1,x_2)\geq |f(x_1)-f(x_2)|$ for all $x_1,x_2\in M$.
   It is then enough to show that for any given curve-flat $1$-Lipschitz function $g\in \Lip_0(M)$, one has $|g(x_1)-g(x_2)|\leq |f(x_1)-f(x_2)|$ for all $x_1,x_2\in M$. 
   Hence, fix such a function $g$ and two points $x_1,x_2\in M$. 
   Since $g$ and $f$ are curve-flat, they are constant on all geodesics of $M$. 
   Therefore, there exist $\tilde{x}_1,\tilde{x}_2\in \mathcal C^\beta\subset M$ such that $g(\tilde{x}_i)=g(x_i)$ and $f(\tilde{x}_i)=f(x_i)$ for $i=1,2$. 
   We may assume then without loss of generality that $x_1$ and $x_2$ belong to $\mathcal C^\beta$ with $x_1<x_2$ (in the real line ordering). 
   Now, using again that $g(M)$ is contained in $g(\mathcal C^\beta)$, we get the estimate
		$$|g(x_1)-g(x_2)|\leq \lambda([g(x_1),g(x_2)])\leq \mathcal{H}^1(([x_1,x_2]\cap C)^\beta)= \mathcal{H}^\beta([x_1,x_2]\cap \mathcal C),$$
		where $([x_1,x_2]\cap \mathcal C)^\beta$ denotes the snowflake of $[x_1,x_2]\cap \mathcal C$, isometrically contained in~$M$. 
  Finally, by \cite[Proposition 5.5]{Cantor}, we have that \[|f(x_1)-f(x_2)|=\lambda(f([x_1,x_2]\cap \mathcal C))=\mathcal{H}^\beta([x_1,x_2]\cap \mathcal C),\] and the conclusion follows.
	\end{remark}

\begin{remark}\label{r:ExampleNotCounter}
Let $M$ be the compact metric space considered in Example~\ref{cor:Example_CF_not_Factp-1-u}.
We claim that $\widehat{f} : \F(M) \to \F(\Real)$ factors through $\ell_1$ for every $f\in \CF_0(M)$.  
In fact, the following much stronger property is true: there exists a compact p-1-u $P\subset M$ (such that $\Free(P)\simeq \ell_1$) and a weak$^*$ continuous extension operator $E:\Lip_0(P) \to \Lip_0(M)$ such that $E\circ R(f))=f$ for every $f\in \CF_0(M)$ where $R:\Lip_0(M) \to \Lip_0(P)$ is the restriction operator. 
It then follows that $\hat{f}=\hat{f}\circ R_*\circ E_*$ where $R_*:\Free(P) \to \Free(M)$ and $E_*:\Free(M) \to \Free(P)$ are the corresponding pre-adjoints.
Indeed, let $P=\mathcal C^\beta$ and define $E$ as the ``affine extension'' operator. In order to check, that this operator is bounded, we observe that $E=\overline{g}^*$ where $g: M \rightarrow \mathcal{F} (\mathcal{C}^\beta)$ is defined as follows:
For each $x \in \mathcal{C}^\beta \subset M$, we let $g(x) := \delta(x)$. 
Next, if $x \in M \setminus \mathcal{C}^\beta$, then $x$ belongs to a unique geodesic which fills a gap of $\mathcal{C}^\beta$, say between the points $x_\ell, x_r \in \mathcal{C}^\beta$. 
In this case, we put $\lambda_x=\frac{d(x,x_\ell)}{d(x_\ell,x_r)}$ and define $g(x) := (1-\lambda_x)\delta(x_\ell) + \lambda_x \delta(x_r)$. 
It is not hard to prove that $g$ is $1$-Lipschitz. 
For instance, let $x \in M \setminus \mathcal{C}^\beta$ and $y \in C^\beta$. 
Suppose that the point $x$ lies in the geodesic $[x_\ell, x_r]$ with $d(x,y) = d(x, x_\ell) + d(x_\ell,y)$. Note that 
\begin{align*}
    \|g(x)-g(x_\ell)\| &= \| (1-\lambda_x ) \delta (x_\ell) + \lambda_x \delta (x_r) - \delta(x_\ell) \| = \lambda_x d(x_\ell, x_r) = d(x,x_\ell) 
\end{align*}
and $\|g(x_\ell) - g(y) \| = d(x_\ell, y)$ since $x_\ell, y \in \mathcal{C}^\beta$. 
Thus, observe that 
\begin{align*}
\|g(x)-g(y) \| &\leq \|g(x)-g(x_\ell)\| + \|g(x_\ell) - g(y) \| = d(x, x_\ell) + d(x_\ell, y) = d(x,y).
\end{align*} 
Where the very last equality follows from our choice of the metric on $M$.
The other cases being similar, 
this proves that $g$ is $1$-Lipschitz; thus its linearization $\overline{g}: \mathcal{F}(M) \rightarrow \mathcal{F}(\mathcal{C}^\beta)$ is a norm-one bounded linear operator.
Finally, $\F(\mathcal C^\beta)$ is isomorphic to $\ell_1$ thanks to \cite[Theorem~8.49]{Weaver2}.

\end{remark}
 	


\subsection{An excursion to the non compact setting}

When $M$ is not necessarily compact, we still have the next elementary but interesting fact.
\begin{lem} \label{LemmaSPU}
	Let $M,N$ be metric spaces and $f:M \to N$ a Lipschitz function which factors through a compact purely 1-unrectifiable metric space $P$, i.e. there are Lipschitz functions $g,h$ such that the following diagram commutes:
			$$ \xymatrix{
		M \ar[rd]_g \ar[rr]^f  & & N  \\
			& P \ar[ur]_{h}  &
		} $$ 

	Then for all $x,y \in M$ such that $f(x)\neq f(y)$ and for all $\varepsilon>0$, there is $\psi \in \lip_0(M)$ with $\norm{\psi}_L \leq 1$ such that $$\displaystyle\abs{\psi(x)-\psi(y)}\geq \frac{d(f(x),f(y))}{\norm{h}_L\norm{g}_L}-\varepsilon.$$
\end{lem}

\begin{proof}
	We first observe that 
	$$d(f(x),f(y))=d(h(g(x)),h(g(y)))\leq \norm{h}_Ld(g(x),g(y)).$$
	Note that $f\neq 0$, so $\|h\|_L>0$ and $\|g\|_L>0$. Since $P$ is purely 1-unrectifiable and compact, by~\cite[Theorem A]{AGPP}, there exists $\varphi \in \lip_0(P)$ with $\norm{\varphi}_L\leq 1$ such that 
	$$\abs{\varphi(g(x))-\varphi(g(y))}\geq d(g(x),g(y))-\varepsilon\|g\|_L.$$
	We set $\psi=\frac{1}{\|g\|_L}\varphi\circ g$ and the conclusion follows.
\end{proof}

Finally, we deal with the case when $M$ is a complete subset of $\R$ below. 

\begin{proposition}
	\label{p:DPonRfactorsp-1-u}
Let $M\subset \Real$ and let $N$ be a metric space, both $M$ and $N$ complete. If $f \in \CF_0 (M,N)$, then $f$ factors through a complete purely 1-unrectifiable metric space.
\end{proposition}

\begin{proof}
	If $M$ is bounded, hence compact, Proposition~\ref{p:cpt-c1r} yields the conclusion.
Next we will deal with the case when $\min M=0=:x_0$ and $\sup M=+\infty$. 
	Let $(x_n) \subset M$ be a strictly increasing sequence with $\lim x_n=\infty$. 
 \begin{figure}[h]
     \centering
     \includegraphics[width=0.75\linewidth]{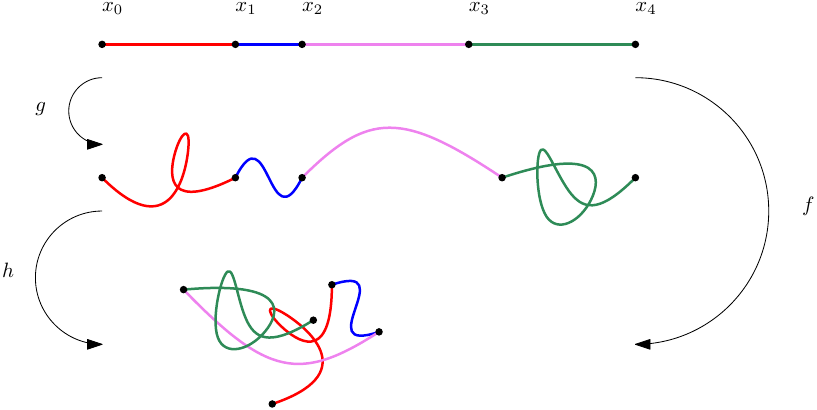}
     \caption{The construction of the factorization is best understood on a picture.}
     \label{fig:UnboundedFactorization}
 \end{figure}
 We denote $M_n=M\cap [x_{n-1},x_n]$ for each $n \in \mathbb{N}$. Now by Proposition~\ref{p:RN_eq_SRN_in_1_rectifiable_spaces} the space $f(M_n)$ is a compact $\H^1$-negligible space. We define a sequence $(P_n)$ of almost disjoint copies of $f(M_n)$. More precisely, for every $n$, let $i_n:P_n \to f(M_n)$ be an isometry such that $i_n^{-1}(f(x_n))=i_{n+1}^{-1}(f(x_n))$ and $P_n \cap P_{n+1}=\set{i_n^{-1}(f(x_n))}$.
	
	We have $f\restricted_{M_n}=i_n\circ (i_n^{-1}\circ f\restricted_{M_n})$. Now we endow the set $P=\bigcup P_n$ with the following distance $d_P$: for $x \in P_n$ and $y \in P_l$ with $n, l \in \mathbb{N}$
		\begin{align*}
			&d_P(x,y) \\ 
			&\, := \begin{cases}
				d(i_n(x),i_n(y)) , &\text{ if }  l = n \\
				d(i_n(x), f(x_n) ) + d(f(x_n), i_{n+1} (y)), &\text{ if }  l = n+1 \\
				d(i_n(x),f(x_n))+\sum_{j=n}^{l-2} d(f(x_{j}),f(x_{j+1}))+d(f(x_{l-1}),i_{l}(y))  
				&\text{ if } l \geq n+2.
			\end{cases}
		\end{align*}

We define $g : M \rightarrow P$ and $h : P \to N$ by 
		\[
		g(x)=i_n^{-1}\circ f(x) \, \text{ if } x \in M_n \,\, \text{ and } \,\,  h(x)=i_n(x) \, \text{ if } x\in P_n.
		\]
		Notice that $g$ and $h$ are well-defined, and $h \circ g = f$. It is not difficult to check that $g$ is $\|f\|_L$-Lipschitz and $h$ is $1$-Lipschitz.

Consider the completion $Q$ of $P$. 
We claim that $Q\setminus P$ is empty or a singleton. 
Indeed, let $(z_n) \subset P$ such that $z_n\to x$.
If $(z_n) \subset \bigcup_{i\in F} P_i$ for a finite set $F$, then $x \in \bigcup_{i\in F} P_i$ by compactness of $ \bigcup_{i\in F} P_i$. 
Thus $x\in P$.
It remains to address the case where no such finite set $F$ exists. In this case, 
we will show that $g(x_n) \to x$ which will finish the proof of the claim. 
Let us denote by $m_n \in \Natural$ one of the indices such that $z_n \in P_{m_n}$.
We may assume, by passing to a subsequence of $(z_n)$, that $(m_n)$ is strictly increasing.
Since $(z_n)$ is Cauchy, passing to a further subsequence, we may assume that $\sum d_P(z_n,z_{n+1})<\infty$. Taking into account that $z_n \in P_{m_n}$ and $z_{n+1} \in P_{m_{n+1}}$, observe from the definition of $d_P$ that 
\begin{align*}
d_P (g(x_{m_n}), g(x_{m_n} +1) + &\cdots + d_P ( g(x_{m_{n+1} -1} ), g(x_{m_{n+1}} )) \\ 
&\leq d_P (z_n, z_{n+2}) \leq d_p (z_n, z_{n+1}) + d_P (z_{n+1}, z_{n+2})
\end{align*} 
for every $n \in \mathbb{N}$. This implies that $\sum d_P(g(x_{n}),g(x_{n+1} )) < \infty$.
In particular, the sequence $(g(x_{n}))$ converges in $Q$. 
On the other hand, since $d_P(z_n,g(x_{m_n}))\leq d_P(z_n,z_{n+1})$ by definition of $d_P$, we conclude that $g(x_{n}) \rightarrow x$ and our claim is proved.

Therefore $Q$ is complete and p-1-u.
Since $\norm{h}_L\leq 1$ and since $N$ is complete, $h$ extends to $Q$ with the same Lipschitz constant.

In the case when $M$ is unbounded below as well as above  we will get by a similar argument that $Q \setminus P$ has at most two elements. The details are left to the reader.
\end{proof}


\subsection{Factoring through Schur spaces and RNP spaces}

We now turn to the ``factorization through Schur or RNP spaces'' problem:

\begin{quest}\label{Question:DP_factors_through_Schur}
	Let $f\colon M\rightarrow N$ be a curve-flat Lipschitz function. 
	\begin{enumerate}
		\item Does it follow that $\widehat{f}$ factors through a Schur space? 
		\item Does it follow that $\widehat{f}$ factors through a space with the RNP?
	\end{enumerate}
\end{quest}

This question remains largely open, we only have a few observations. Of course, as we already mentioned, if $f$ factors through a complete p-1-u metric space then $\widehat{f}$ factors through a Schur space with RNP. Still in the way of positive results towards Question~\ref{Question:DP_factors_through_Schur}, we have the following ``reduction to bounded sets" principle which we will apply to several cases.

\begin{proposition}\label{prop:reduction_bdd_sets}
	Let (P) be any property stable under $\ell_1$ sums (such as the Schur property or the RNP). Let $f \in \Lip_0(M,N)$ such that there exists $C>0$ for which, for every bounded set $B \subset M$, $\widehat{f}\restricted_{\F_M(B)}$ factors through a space $Z$ that has the property (P) with $\widehat{f}\restricted_{\F_M(B)} = \phi \circ q$ where $\phi:Z \to \F (N)$ and $q : \F_M (B) \to Z$ satisfy $\|\phi\|\|q\| \leq C$.
 Then $\widehat{f}$ factors through a space with property (P).
\end{proposition}

The proof is based on Kalton's decomposition \cite[Section~4]{Kalton04}. In fact, we will use the similar construction presented in \cite[Section~2.2]{AP23}. For every $n \in \Z$ and $x \in M$, let 
$$
\Lambda_n(x) = \left\{
\begin{array}{ll}
	0 & \text{ if } d(x,0) \leq 2^{n-1}\\
	2^{-n+1}d(x,0) - 1 & \text{ if } 2^{n-1} \leq d(x,0) \leq 2^{n}\\
	2-2^{-n}d(x,0) & \text{ if } 2^n \leq d(x,0) \leq 2^{n+1}\\
	0 & \text{ if } d(x,0) \geq 2^{n+1}.
\end{array}
\right.
$$
Note that $\| \Lambda_n \|_L = \frac{1}{2^{n-1}}$. Let us denote $M_n = B(0,2^{n+1})\setminus B(0, 2^{n-1})$ so that $\supp(\Lambda_n) \subset M_n$. Next, for each $n \in \Z$, we introduce the pointwise multiplication operator $T_n : \Lip_0(M) \to \Lip_0(M)$ given by $T_n (f)(x)=f(x)\Lambda_n(x)$.
As explained in \cite[Section~2.2]{AP23}, $T_n$ defines a weak$^*$-to-weak$^*$ continuous linear operator from $\Lip_0(M)$ into $\Lip_0(M)$ with $\|T_n\| \leq 5$.  We will denote by $W_n \colon \F(M) \to \F(M)$ its pre-adjoint operator. Notice that: 
$$\forall x \in M, \quad W_n(\delta(x)) = \Lambda_n(x) \delta(x).$$ 
Since $\supp(\Lambda_n) \subset M_n$, it is readily seen that $\supp(W_n(\gamma)) \subset M_n$ for every $\gamma \in \F(M)$. Therefore we can view $W_n$ as an operator from $\F(M)$ to $\F_M(M_n) \equiv \F(M_n \cup \{0\})$.  Now for every $\gamma \in \F(M)$, one has 
$$ \sum_{n \in \Z} \|W_n \gamma \| \leq 45 \|\gamma\| \; \text{ and moreover } \; \gamma = \sum_{n \in \Z} W_n \gamma $$
where the convergence of the series is in norm. 
So  $T : \F(M) \to \left(\sum_{n \in \Z} \F_M(M_n)\right)_{\ell_1}$ given by 
$$ \forall \gamma \in \F(M), \quad T(\gamma) = (W_n(\gamma))_{n \in \Z},$$
defines a bounded operator with norm $\|T\| \leq 45$.

\begin{proof}
	Using the notation introduced above, we define $S : \Big(\sum_{n \in \Z} \F_M(M_n)\Big)_{\ell_1} \to \ell_1(\F(N))$ by:
	$$ \forall (\gamma_n)_n \in \Big(\sum_{n \in \Z} \F_M(M_n)\Big)_{\ell_1}, \quad S((\gamma_n)_n) := \Big( \widehat{f}\restricted_{\F_M(M_n)}(\gamma_n) \Big)_n.$$
	By assumption, each map $\widehat{f}\restricted_{\F_M(M_n)} $ factors through a space $Z_n$ with property (P):
				$$ \xymatrix{
		\F_M(M_n) \ar[rd]_{q_n} \ar[rr]^{\widehat{f}\restricted_{\F_M(M_n)}}  & & \mathcal{F}(N)  \\
			& Z_n \ar[ur]_{\phi_n}  &
		} $$ 
 with $\|\phi_n\|\|q_n\| \leq C$. Of course, after rescaling, we may as well assume that $\|q_n\|\le 1$ and $\|\phi_n\|\le C$. We readily deduce the next factorization for $S$:
	$$ \xymatrix{
		\Big(\sum_{n \in \Z} \F_M(M_n)\Big)_{\ell_1} \ar[rd]_{(q_n)_n} \ar[rr]^{S}  & &\ell_1(\F(N)) \\
			&  \Big(\sum_{n \in \Z} Z_n \Big)_{\ell_1} \ar[ur]_{(\phi_n)_n}  &
		} $$ 
 where the operators $(q_n)_n$ and $(\phi_n)_n$ are well-defined bounded linear operators, with $\|(q_n)_n\|\le 1$ and $\|(\phi_n)_n\|\le C$. Notice that $\big(\sum_{n \in \Z} Z_n \big)_{\ell_1}$ has the property (P), since this property is stable under $\ell_1$-sums. 
	To conclude, we introduce $R : \ell_1(\F(N)) \to \F(N)$ given by $R((\mu_n)_n) = \sum_n \mu_n$. It is an easy exercise to check that $\widehat{f} = R \circ S \circ T = R \circ (q_n)_n \circ (\phi_n)_n \circ T$.
\end{proof}

Practically speaking, in the Propositions \ref{p:cpt-ulf} and \ref{p:cpt-c1r} concerning $Fact_{p1u}$, we haven't explicitly mentioned it, but the factorizations are actually uniformly bounded: Let $M$ be a proper metric space and $N$ any metric space. If $f \in \lip_0 (M,N)$, then $\hat{f}$ factors through a Schur space with RNP. Indeed, let $B \subseteq M$ be a bounded set. Then $B$ is compact and $f \restricted_{B} \in \lip_0 (B,N)$. Thanks to Proposition \ref{p:cpt-ulf}, we have that the following diagram commutes:
$$ \xymatrix{
		B \ar[rd]_{\pi_B} \ar[rr]^{f\restricted_{B}}  & & N \\
			& B_{\lip}  \ar[ur]_{(f \restricted_B)_{\lip}}  &
		} $$ 
Notice that $\|\pi_B\| \leq 1$, $\| (f \restricted_B)_{\lip} \|_L  \leq \|f\|_L$ (the upper bounds are not depending on the choice of $B$), and $B_{\lip}$ is p-1-u. Now, Proposition~\ref{prop:reduction_bdd_sets} concludes that $\hat{f}$ factors through a Schur space with RNP.

On the other hand, suppose that $M$ is a proper metric space which is countably $1$-rectifiable and $f \in \CF_0 (M,N)$. Then $\hat{f}$ does not fix a copy of $L_1$; so by Remark~\ref{Remark:H1_negligible_DP} we obtain that $f(M)$ is $\mathcal{H}^1$-negligible. It follows that for any bounded subset $B$ of $M$, the set $f(B)$ is compact and $\mathcal{H}^1$-negligible. Therefore, $f$ satisfies the assumption in Proposition \ref{prop:reduction_bdd_sets}. Summarizing, we observe the following corollary. 

\begin{cor}
    Let $M$ be a proper metric space and $N$ any metric space.
        \begin{enumerate}
        \item If $f \in \lip_0 (M,N)$, then $\hat{f} : \mathcal{F}(M) \rightarrow \mathcal{F}(N)$ factors through a Schur space with RNP. 
        \item Assuming $M$ is countably $1$-rectifiable, if $f \in \CF_0 (M,N)$, then $\hat{f}$ factors through a Schur space with RNP. 
    \end{enumerate}
\end{cor}


\section*{Appendix: Short proof of Theorem A in \texorpdfstring{\cite{AGPP}}{[5]}}

Theorem~A in \cite{AGPP} is used repeatedly along this article. Its original proof in \cite{AGPP} is fairly complex, and it uses some underlying principles previously used to show Lemma~3.4 in~\cite{Bat20}. 
In this appendix, we include a short proof of Theorem~A in \cite{AGPP} based directly on Lemma~3.4 in \cite{Bat20}. This proof was communicated to the authors of~\cite{AGPP} by Sylvester Eriksson-Bique. We are grateful to him for allowing us to include it in this paper.

\begin{prop}\label{p:Eriksson-Bique}
Let $M$ be a compact purely $1$-unrectifiable metric space. Then $B_{\lip_0(M)}$ is w$^*$-dense in $B_{\Lip_0(M)}$.
\end{prop}
\begin{proof}
    
Let $f$ in $B_{\Lip_0(M)}$ and $\varepsilon > 0$ be arbitrary. Set $f_0 := f$ and $\theta_0 := diam(M)$. Since $M$ is purely $1$-unrectifiable, $f_0$ is curve-flat. 
Applying Lemma~\ref{Lemma_3.4_Bate}, as we did in the proof of Lemma~\ref{Lemma_Bate}, we obtain a function $f_1\in B_{\text{lip}_0^{\frac{1}{2}} }(M) $ such that $\|f_0-f_1\|_\infty< \frac{\varepsilon}{2}$ and $\frac{|f_1(x)-f_1(y)|}{d(x,y)}\leq \frac{|f_0(x)-f_0(y)|}{d(x,y)}+ \frac{1}{2}$ for every $x\neq y \in M$. 
Denote by $\theta_1$ the radius at which $f_1$ is $\frac{1}{2}$-flat. Importantly, $f_1$ is again curve-flat since $M$ is purely $1$-unrectifiable.

Continuing inductively in this way, we obtain sequences $(f_n)_n\subset B_{\Lip_0(M)}$ and $(\theta_n)_n$ such that for all $n\in\mathbb{N}$ and all $x\neq y \in M$:
\begin{enumerate}
    \item $f_n$ is $\frac{1}{2^n}$-flat at radius $\theta_n$, 
    \item $\|f_n-f_{n-1}\|_\infty < \frac{\varepsilon}{2^n}$, and
    \item $\frac{|f_{n}(x)-f_{n}(y)|}{d(x,y)}\leq \frac{|f_{n-1}(x)-f_{n-1}(y)|}{d(x,y)}+\frac{1}{2^n}$ for every $x\neq y \in M$.
\end{enumerate}

Let $f_\infty$ be the limit of $(f_n)_n$ in the topology of uniform convergence. In particular, $(f_n)_n$ converges to $f_\infty$ in the weak$^*$-topology and $\|f-f_\infty\|_\infty<\varepsilon$.

It only remains to show that $f_\infty$ is uniformly locally flat. Let $x\neq y\in M$ and fix an arbitrary $n\in\N$. For any $m \geq n$, we have by induction
$$\frac{|f_m(y)-f_m(x)|}{d(x,y)} \leq \frac{|f_n(y)-f_n(x)|}{d(x,y)}+\frac{1}{2^n}.$$ 
Taking $m \to \infty$ yields $\frac{|f_\infty(y)-f_\infty(x)|}{d(x,y)} \leq \frac{|f_n(y)-f_n(x)|}{d(x,y)}+\frac{1}{2^n}$. Notice that this implies that $\|f_\infty\|_L \leq 1$. Morevoer, since $f_n$ is $\frac{1}{2^n}$-flat at radius $\theta_n$, we obtain that $f_\infty$ is $\frac{1}{2^{n-1}}$-flat at radius $\theta_n$. As it holds for arbitrary $n \in \mathbb{N}$, the map $f_\infty$ is uniformly locally flat.
\end{proof}

Let us note that this approach does not work to obtain Theorems \ref{t:DP} and \ref{t:RN}. Particularly, one may attempt to use a similar argument in Propositions \ref{prop:CFimpliesDPCompactCase} and \ref{prop:CFimpliesRNCompactCase}, approximating a curve-flat function by flatter and flatter functions, and hoping to take the uniform limit of a sequence of such approximations to obtain a uniformly locally flat function, whose linearization can easily be seen to be DP and RN respectively (indeed, when $M$ is compact it is a compact operator as shown in~\cite{ACP21,Jimenez,Kamowitz}). However, we crucially need each approximation to be again curve-flat in order to continue the induction, which need not be true outside of the purely $1$-unrectifiable domain case. In fact, this approach simply cannot work, since by Example 5.25 in \cite{AGPP} (also used in Example~\ref{cor:Example_CF_not_Factp-1-u}) there exists a compact metric space that admits a non-trivial curve-flat function, but where only constant functions are uniformly locally flat. Therefore, such a curve-flat function cannot be approximated at all by a uniformly locally flat function. 

\section*{Competing interests}
No competing interest is declared.

\section*{Author contributions statement}
All authors discussed the results and contributed to the final manuscript.

\section*{Acknowledgments}

We are grateful to Sylvester Eriksson-Bique for letting us include his short proof of Proposition~\ref{p:Eriksson-Bique}. {We also thank both referees for their valuable comments and for their careful and thoughtful reading of the manuscript.} G. Flores was supported by Universidad de O'Higgins (112011002-PD-17706171-9).
M. Jung was supported by June E Huh Center for Mathematical Challenges (HP086601)
at Korea Institute for Advanced Study.
G. Lancien, C. Petitjean, A. Proch\'azka and A. Quilis were partially supported by the French ANR project No. ANR-20-CE40-0006. A. Quilis was also supported by MCIN/AEI/10.13039/501100011033 (PID2021-122126NB-C33) and CNRS project PEPS JCJC 2024.


\end{document}